\documentclass[11pt,a4paper,dk,reqno]{amsart}
\usepackage{amsmath,amsfonts,amssymb,amsthm,amscd}
\usepackage[english]{babel}
\usepackage[mathscr]{eucal}
\usepackage{graphicx}
\usepackage{appendix}

\usepackage[margin=3cm]{geometry}

\DeclareMathAlphabet{\mathpzc}{OT1}{pzc}{m}{it}

\numberwithin{equation}{section}

\newtheorem{thm}{Theorem}[section]

\newtheorem{remark}{Remark}[section]

\newtheorem{theoreme}{Theorem}[section]

\newtheorem{proposition}{Proposition}[section]
\newtheorem{corollaire}{Corollary}[section]
\newtheorem{lemme}{Lemma}[section]

\newcommand{\dsp}{\displaystyle}
\newcommand{\eps}{\varepsilon}
\newcommand{\RN}{\mathbb{R}^N}

\newcommand{\R}{\mathbb{R}}
\newcommand{\dt}{\partial_t}
\newcommand{\loc}{\mathrm{loc}}
\newcommand{\CC}{\mathbb{C}^2}

\newcommand*{\ce}{\subset\joinrel\joinrel\joinrel_{\mbox{\lower .28ex%
\hbox{$\rightarrow$}}}}

\newcommand{\ke}{\kappa}
\newcommand{\lae}{\nu_{\varepsilon}}

\newcommand{\om}{\omega}
\newcommand{\la}{\lambda}
\newcommand{\D}{\Delta}

\newcommand{\C}{\mathbb{C}}
\newcommand{\diver}{\textrm{div}}

\newcommand{\FF}{\mathcal{F}^{-1}}

\newcommand{\G}{\Gamma}
\newcommand{\vor}{(1+\frac{\eps}{\sqrt{2}}b)}
\newcommand{\x}{(c,d)}

\newcommand{\bb}{\textbf{\textrm{b}}}
\newcommand{\vv}{\textbf{\textrm{v}}}

\newcommand{\U}{(b,z)}

\newcommand{\finpreuve}{\hfill$\Box$}

\begin{document}


\title[Damped wave dynamics for a GL equation with low dissipation]
{Damped wave dynamics for a complex Ginzburg-Landau equation with low dissipation}
\author{Evelyne Miot}

\address[E. Miot]{
Dipartimento di Matematica G. Castelnuovo\\ Universit\`a di Roma ``La Sapienza''\\Italy} 
 \email{miot@ann.jussieu.fr}

\date{\today}
\maketitle

\begin{abstract}
 We consider a complex Ginzburg-Landau equation on $\RN$, corresponding to a Gross-Pitaevskii equation 
with  a small dissipation term. We study an asymptotic regime for long-wave perturbations of constant 
maps of modulus one. We show that such solutions never vanish on $\RN$ and we derive a 
damped wave dynamics for the perturbation. Our results are obtained in the same spirit
as those by Bethuel, Danchin and Smets for the Gross-Pitaevskii equation \cite{BDS}.
\end{abstract}


\section{Introduction}

We consider a complex Ginzburg-Landau equation 
\begin{equation}\tag{C}
\label{eq : cgl} \partial_{t} \Psi=(\kappa+i)[\D
\Psi+\Psi(1-|\Psi|^2)],
\end{equation} 
where $\Psi=\Psi(t,x):\R_+\times\R^N\to\C$, with $N\geq 1$, is a complex-valued map and where
$0<\kappa<1$. 

Equation \eqref{eq : cgl} admits elementary non-vanishing solutions, which are given by all constant maps of modulus equal to 
one. The aim of this paper is to study the dynamics for \eqref{eq : cgl} near such states. We focus on a regime in which the solutions $\Psi$ do not vanish on $\RN$, so that we may 
write them into the form 
 $$\Psi=r\exp(i\phi).$$ 
Secondly, we assume that $(r^2,\nabla\phi)$ is a long-wave perturbation of $(1,0)$. More precisely, we introduce a 
small parameter $\eps>0$ and we define $(r^2,\nabla \phi)$ through the change of variables
\begin{equation}
\label{pert-eq:pert}
\begin{cases}
\dsp r^2(t,x)=1+\frac{\eps}{\sqrt{2}}a_\eps(\eps t,\eps x)\\
2\nabla \phi(t,x)=\eps u_\eps(\eps t,\eps x),
\end{cases}
\end{equation}
where $(a_\eps,u_\eps)$ belongs to $C(\R_+,H\sp {s+1}\times H^s)$, with $s\geq 2$, 
and satisfies suitable bounds.

Our objective is two-fold. First, to define $(a_\eps,u_\eps)$ we wish to determine how long 
a solution initially given by \eqref{pert-eq:pert} does not vanish on $\RN$. Our 
second purpose is to investigate the dynamics of $(a_\eps,u_\eps)$ when $\eps$ vanishes and $\kappa$ is small. This
asymptotic dynamics depends on the balance between the amount $\kappa$ of dissipation in Eq. \eqref{eq : cgl}
 and the size $\eps$ of the perturbation; to characterize this balance we introduce the ratio
$$\lae=\frac{\ke}{\eps}.$$

According to \eqref{eq : cgl} we obtain the equations for the perturbation $(a_\eps,u_\eps)$ 
\begin{equation}
 \label{eq : a-u}
\begin{cases}
 \dsp\dt a_\eps+\sqrt{2}\diver u_\eps+2\lae-\ke \eps\D a_\eps={\text{f}}_\eps(a_\eps,u_\eps)\\
\dsp \dt u_\eps+\sqrt{2}\nabla a_\eps-\ke \eps\D u_\eps={\text{g}}_\eps(a_\eps,u_\eps),
\end{cases}
\end{equation}
where ${\textrm{f}}_\eps$ and ${\textrm{g}}_\eps$ are given by
\begin{equation}
\label{syst:citta}
\begin{cases}
 \dsp \dsp {\text{f}}_\eps(a_\eps,u_\eps)=\sqrt{2}\ke\Big(-2|\nabla \rho_a|^2-\rho_a^2\frac{|u_\eps|^2}{2}-a_\eps^2\Big)-\eps \diver(a_\eps u_\eps)\\
\dsp {\text{g}}_\eps(a_\eps,u_\eps)=\ke \eps \nabla\left(\frac{\nabla \rho_a^2}{\rho_a^2}\cdot u_\eps\right)+2\eps \nabla \frac{\D \rho_a}{\rho_a} -\eps u_\eps\cdot\nabla u_\eps,
\end{cases}
\end{equation}
with  $$\rho_a^2(t,x)=1+\frac{\eps}{\sqrt{2}}a_\eps(t,x).$$

\medskip

Our first result establishes that if the initial perturbation is not too large, 
the solution $\Psi$ never exhibits a zero so that \eqref{pert-eq:pert} does hold for all time. 
\begin{thm}
\label{thm : theorem2}
Let $s$ be an integer such that $s> 1+N/2$. There exist positive 
numbers $K_1(s,N)$, $K_2(s,N)$  and $0<\kappa_0(s,N)<1$, depending only on $s$ and $N$,
 satisfying the following property.

Let $0<\kappa\leq \kappa_0(s,N)$.
For $0<\eps\leq 1$, let $(a_\eps^0,\varphi_\eps^0)\in H^{s+1}(\RN)^2$ such that 
$$M_0:=\|(a_\eps^0,u_\eps^0)\|_{H^{s}}+\eps  \|a_\eps^0\|_{H^{s+1}}+\|\varphi_\eps^0\|_{L\sp 2}\leq \frac{\min(\lae,\ke^{-1},\eps^{-1})}{K_1(s,N)}, $$
where $u_\eps^0=2\nabla \varphi_\eps^0$.

Then Eq. \eqref{eq : a-u}-\eqref{syst:citta} has a unique global solution $(a_\eps,u_\eps)$ in $C(\R_+,H^{s+1}\times H\sp s)$ such that $(a_\eps,u_\eps)(0)=(a_\eps^0,u_\eps^0)$. Moreover
\begin{equation*}
\|(a_\eps,u_\eps)\|_{L^\infty(H\sp s)}+\eps \|a_\eps\|_{L^\infty(H^{s+1})}\leq K_2(s,N)M_0.
\end{equation*}
Finally, if $\Psi$ denotes the corresponding solution to Eq. \eqref{eq : cgl}, we have for all $t\geq 0$
\begin{equation*}
 \left\||\Psi(t)|^2-1\right\|_{\infty}<\frac{1}{2}.
\end{equation*}
\end{thm}

\medskip

\begin{remark} Fixing $\kappa=\kappa_0$ and $\eps=\eps_0$,
 Theorem \ref{thm : theorem2} entails that for initial data 
\begin{equation*}
 \Psi^0(x)=\left(1+\tilde{a}^0(x)\right)^{1/2}\exp(i\tilde{\varphi}^0(x)),
\end{equation*}
with $\|(\tilde{a}^0,\tilde{\varphi}^0)\|_{H^{s+1}}\leq C$, where $C$ only depends on $s$ and $N$, the corresponding solution\footnote{Given by Theorem \ref{thm : cauchy} below.} $\Psi$ to Eq. \eqref{eq : cgl} remains bounded and bounded away from 
zero for all time.
\end{remark}

\begin{remark} For all $0<\eps\leq \eps_0$ and $0<\kappa\leq \kappa_0$ satisfying $\eps\leq \ke$, so that $\lae\geq 1$, 
Theorem \ref{thm : theorem2} allows 
to handle initial data 
\begin{equation}
\label{def:initial}
 \Psi_\eps^0(x)=\left(1+\frac{\eps}{\sqrt{2}}a^0(\eps x)\right)^{1/2}\exp(i\varphi^0(\eps x)),
\end{equation}
where $(a^0,\varphi^0)\in H^{s+1}(\RN)^2$ does not depend on $\eps$, so that $M_0$ is constant, and where $M_0$ is smaller than a number depending only on $s$ and $N$.
\end{remark}

\medskip

Once the question of existence for $(a_\eps,u_\eps)$ has been settled, our next task is to determine a 
simplified system of equations to describe its asymptotic dynamics. 
From now on we focus on a regime with low dissipation, namely 
we further assume that 
\begin{equation*}
 \kappa=\kappa(\eps)\quad \text{and} \quad \lim_{\eps\to 0} \kappa(\eps)=0.
\end{equation*}
In view of \eqref{syst:citta}, this is a natural ansatz in order to treat the second 
members ${\text{f}}_\eps$ and ${\text{g}}_\eps$ as perturbations in the limit $\eps \to 0$. 
 Eq. \eqref{eq : a-u} then formally reduces to
a damped wave equation 
\begin{equation}
 \label{eq : limit-wave}
\begin{cases}
\dsp \dt a+\sqrt{2}\diver u+2\lae a=0\\
\dsp \dt u+\sqrt{2}\nabla a=0,
\end{cases}
\end{equation}
 with propagation speed equal to $\sqrt{2}$ and damping coefficient equal to $2\lae$.

\medskip

As a consequence of Theorem \ref{thm : theorem2} we 
can compare the solution $(a_\eps,u_\eps)$ to the one of the linear damped wave equation \eqref{eq : limit-wave} with loss of three derivatives.
\begin{thm}
 \label{thm : comparison 1}
Let $s$ be an integer such that $s>1+N/2$. Let $(a_\eps^0,\varphi_\eps^0)\in H^{s+1}(\RN)^2$ satisfy the assumptions of Theorem \ref{thm : theorem2}. Let $u_\eps^0=2\nabla \varphi_\eps^0$. 

We denote by $(a_\ell,u_\ell)\in C(\R_+,H^{s+1}\times H^s)$ the solution of Eq. \eqref{eq : limit-wave} with initial datum $(a_\eps^0,u_\eps^0)$. 

There exists a constant $K_3(s,N)$ depending only on $s$ and $N$ such that for all $t\geq 0$
\begin{equation*}
 \|(a_\eps-a_\ell,u_\eps-u_\ell)(t)\|_{H^{s-2}}\leq K_3(s,N)  (\eps \ke t)^{1/2}\max(1,\lae^{-1})( M_0^2+M_0),
\end{equation*}
where $M_0$ is defined in Theorem \ref{thm : theorem2}.
\end{thm}
In particular, for initial data given by \eqref{def:initial}, the approximation by the damped wave equation is 
optimal when $\ke$ and $\eps$ are comparable. Moreover, Theorem \ref{thm : comparison 1}
yields a correct 
approximation up to times of order $C(\ke \eps)^{-1}$. In order to handle 
larger times, it is helpful 
to take into account the linear parabolic terms in \eqref{eq : a-u}:
\begin{equation}
 \label{eq : a-u homogeneous}
\begin{cases}
\dsp \dt a+\sqrt{2}\diver u+2\lae a-\ke  \eps\D a=0\\
\dsp \dt u+\sqrt{2}\nabla a-\ke \eps\D u=0.
\end{cases}
\end{equation}
Our next result presents uniform in time comparison estimates with the solution of Eq. \eqref{eq : a-u homogeneous}
for high order derivatives.
\begin{thm}
 \label{thm : comparison 2}
Let $s$ be an integer such that $s>1+N/2$. Let $(a_\eps^0,\varphi_\eps^0)\in H^{s+1}(\RN)^2$ satisfy the assumptions of Theorem \ref{thm : theorem2}. 

We denote by $(a_\ell,u_\ell)\in C(\R_+,H^{s+1}\times H^s)$ the solution of 
Eq. \eqref{eq : a-u homogeneous} with initial datum $(a_\eps^0,u_\eps^0)$. 

There exists a constant $K_4(s,N)$ depending only on $s$ and $N$ such that 
\begin{equation*}
\begin{split}
&\bullet \|(a_\eps-a_\ell,u_\eps-u_\ell)\|_{L^\infty(H\sp {s-2})}\leq K_4(s,N)
\left(\ke \max(1,\lae^{-1})^2M_0^2+\eps  \max(1,\lae^{-1})M_0\right),\\
&\bullet \|(a_\eps-a_\ell,u_\eps-u_\ell)\|_{L^\infty(H\sp {s-1})}\leq K_4(s,N)
\Big(\max(1,\lae^{-1})\big(\max(\ke,\eps)+\lae^{-1}\big)M_0^2+\lae^{-1}M_0\Big),\\
&\bullet \|(a_\eps-a_\ell,u_\eps-u_\ell)\|_{L^\infty(H\sp {s})}\leq 
K_4(s,N)\Big((\lae^{-1}\max(1,\lae^{-1})+\ke^{-1})M_0^2+\ke^{-1}M_0\Big).  \\
&\quad \text{Finally, for all $t\geq 0$}\\
&\bullet \|(a_\eps-a_\ell,u_\eps-u_\ell)(t)\|_{H\sp {s-2}}\leq 
K_4(s,N) (\eps \ke t)^{1/2}\left(\max(1,\lae^{-1}) M_0^2+\lae^{-1}M_0\right),\\
&\bullet \|(a_\eps-a_\ell,u_\eps-u_\ell)(t)\|_{H\sp {s-1}}\leq K_4(s,N) (\eps \ke^{-1}t)^{1/2} M_0.\\
\end{split}
\end{equation*}
\end{thm}

We come back to initial data given by \eqref{def:initial}. Since $\kappa^{-1}$ diverges when $\eps\to 0$, Theorem
\ref{thm : comparison 2} does not provide a correct approximation for $s$-order derivatives. 
However, Eq. \eqref{eq : a-u homogeneous} yields a satisfactory large in time approximation for the  
derivatives of order $s-1$ if $\lae^{-1}$ vanishes with $\eps$. In fact, the corresponding  
comparison estimate is optimal whenever $\kappa$ and 
$\sqrt{\eps}$ are proportional.
This is due to the fact that the regularizing properties of the parabolic contributions in \eqref{eq : a-u homogeneous} 
become less efficient when $\kappa$ is small. On the other hand, as in Theorem \ref{thm : comparison 1}, the global 
in time comparison estimates involving the lower $(s-2)$-order derivatives are 
more efficient when $\ke$ and $\eps$ are proportional. 

\bigskip

The complex Ginzburg-Landau equations are widely used in the physical literature as a model for various phenomena such as 
superfluidity,
Bose-Einstein condensation or superconductivity, see \cite{aranson}. In the specific form considered here, Eq. 
\eqref{eq : cgl} corresponds to a dissipative extension of the purely dispersive Gross-Pitaevskii equation
\begin{equation}\tag{GP}\label{eq:gross-pita}\partial_{t} \Psi=i[\D
\Psi+\Psi(1-|\Psi|^2)].\end{equation}
A similar asymptotic regime for \eqref{eq:gross-pita} has been recently investigated by Bethuel, Danchin and Smets \cite{BDS}. The analysis of \cite{BDS} exhibits a lower bound for the first time $T_\eps$ where the
solution vanishes and shows that $(a_\eps,u_\eps)$ essentially behaves according to 
the free wave equation ($\lae\equiv 0$), or to a similar version, until then.

In the two-dimensional case $N=2$, there exists a formal analogy between Eq. \eqref{eq : cgl} and the 
Landau-Lifschitz-Gilbert equation for sphere-valued magnetizations in three-dimensional ferromagnetics,
see \cite{capella, americains2}. We mention that a thin-film regime leading to a damped wave dynamics for the 
in-plane components of the magnetization has been studied by
Capella, Melcher and Otto \cite{capella}.

Finally, still in the two-dimensional case $N=2$, Eq. \eqref{eq : cgl} presents another remarkable regime in 
which the solutions exhibit zeros (vortices). This regime has been investigated  by Kurtzke, Melcher, Moser and Spirn 
\cite{americains} and the author \cite{miot} when $\ke$ is proportional
to $|\ln \eps|^{-1}$. In this setting, Eq. \eqref{eq : cgl} is considered under the form  
\begin{equation}\tag{C$_{\eps}$}
\label{eq : cgle} \dt \Psi_\eps=(\ke+i)[\D \Psi_\eps+\frac{1}{\eps^2}\Psi_\eps(1-|\Psi_\eps|^2)],
\end{equation}
which is obtained from the original equation via the parabolic scaling 
\begin{equation}
\label{def:scaling}
 \Psi_\eps(t,x)=\Psi\left(\frac{t}{\eps^2},\frac{x}{\eps}\right).
\end{equation}
A natural extension of the results in \cite{americains,miot} would consist in allowing
for superpositions of vortices and oscillating phases in the initial data. This difficult issue was a strong motivation to analyze 
the behavior of the phase in the regime \eqref{pert-eq:pert}, excluding vortices, as a first attempt 
to tackle the general situation where it is coupled with vortices.

 
\section{General strategy}

We now present our approach
for proving Theorems \ref{thm : theorem2}, \ref{thm : comparison 1} and \ref{thm : comparison 2}, 
which will be partly borrowed from the analysis in \cite{BDS} for the Gross-Pitaevskii equation.

\medskip

First, we handle Eq. \eqref{eq : cgl} in its parabolic scaling \eqref{def:scaling} yielding Eq. \eqref{eq : cgle}. We define the variables
 \begin{equation*}
 \begin{cases}
\dsp  b_\eps(t,x)=a_\eps\left(\frac{t}{\eps},x\right)\\
\dsp v_\eps(t,x)=u_\eps\left(\frac{t}{\eps},x\right),
 \end{cases}
\end{equation*}
so that in the regime \eqref{pert-eq:pert} we have
\begin{equation}
\label{def:no-zero}
 \Psi_\eps(t,x)=\rho_\eps(t,x)\exp(i\varphi_\eps(t,x))\quad \text{on}\quad \R_+\times \RN,
\end{equation}
where 
\begin{equation}
\label{ultimate}
\begin{cases}
\dsp  \rho_\eps^2(t,x)=1+\frac{\eps}{\sqrt{2}}b_\eps(t,x)\\
\dsp 2\nabla \varphi_\eps(t,x)=\eps v_\eps(t,x).
\end{cases}
\end{equation}

\medskip

The system for $(b_\eps,v_\eps)$ translates into 
\begin{equation}
 \label{eq : b-v}
\begin{cases}
 \dsp\dt b_\eps+\frac{\sqrt{2}}{\eps}\diver v_\eps+\frac{2\lae}{\eps}b_\eps-\ke\D b_\eps=\tilde{f}_\eps(b_\eps,v_\eps)\\
\dsp \dt v_\eps+\frac{\sqrt{2}}{\eps}\nabla b_\eps-\ke\D v_\eps=\tilde{g}_\eps(b_\eps,v_\eps),
\end{cases}
\end{equation}
where
\begin{equation}
\label{pert-eq:tilde}
\begin{cases}
\dsp \tilde{f}_\eps(b_\eps,v_\eps)=\sqrt{2}\lae \left(-2|\nabla \rho_\eps|^2-\rho_\eps^2\frac{|v_\eps|^2}{2}-b_\eps^2\right)- \diver(b_\eps v_\eps)\\
\dsp \tilde{g}_\eps(b_\eps,v_\eps)=\ke  \nabla\left(\frac{\nabla \rho_\eps^2}{\rho_\eps^2}\cdot v_\eps\right)+2 \nabla
 \left(\frac{\D \rho_\eps}{\rho_\eps}\right) - v_\eps\cdot\nabla v_\eps.
\end{cases}
\end{equation}

\bigskip

For  a map $\Psi \in H^1_\loc$, the Ginzburg-Landau energy of $\Psi$ is defined by
\begin{equation*}
 E_\eps(\Psi)=\int_{\RN}\Big( \frac{|\nabla \Psi|^2}{2}+\frac{(1-|\Psi|^2)^2}{4\eps^2}\Big)\,dx,
\end{equation*}
and $\mathcal{E}$ denotes the corresponding space of finite energy fields.
For the Gross-Pitaevskii equation the Ginzburg-Landau energy is an Hamiltonian, whereas for solutions to Eq. \eqref{eq : cgle}
it decreases in time.
Note that, in the regime \eqref{def:no-zero}-\eqref{ultimate}, the solution $\Psi_\eps$ belongs to $\mathcal{E}$ since 
 $(b_\eps,v_\eps)\in H^1\times L^2$. In fact, one has
$$E_\eps(\Psi_\eps)\simeq C(\|(b_\eps,v_\eps)\|_{L^2}^2+\eps^2\|\nabla b_\eps\|_{L^2}^2)$$ provided that 
$\||\Psi_\eps|-1\|_{\infty}<1$. 

\medskip

Our first issue is to solve the Cauchy problem for \eqref{eq : cgle} so that $(b_\eps,v_\eps)$ being 
defined by \eqref{ultimate}, as long as $\Psi_\eps$ does not vanish, does belong to $C(H\sp{s+1}\times H\sp s)$. 
As mentioned, the initial field $\Psi_\eps^0$ has finite Ginzburg-Landau energy. 
In \cite{gallo} (see also \cite{gerard}) it has been shown that
 $$\mathcal{E}\subset\mathcal{W}+H^1(\RN).$$
Here the space $\mathcal{W}$, which will be defined in Section \ref{section : cauchy} below, contains 
in particular all constant maps of modulus one. 
It is therefore natural to determine the solution $\Psi_\eps$ in $C(\mathcal{W}+H^{s+1})$. 
This is done in Section \ref{section : cauchy}.

\medskip

In Theorems \ref{thm : theorem2}, \ref{thm : comparison 1} and \ref{thm : comparison 2} one assumes 
that $\|b_\eps^0\|_{\infty}$ is bounded in such a way that $|\Psi_\eps^0|$ is bounded and bounded away from zero. 
More precisely, the constant $K_1(s,N)$ can be adjusted so that
\begin{equation}
 \label{est:binf}
c(s,N)\frac{\eps}{\sqrt{2}} \|b_\eps^0\|_{H^s}<\frac{1}{2}.\end{equation}
Here the constant $c(s,N)$ corresponds to the 
Sobolev embedding $H^s(\RN)\subset L^\infty(\RN)$ for $s>N/2$. Hence \eqref{est:binf}  guarantees that $\||\Psi_\eps^0|^2-1\|_{\infty}<1/2$. 

As long as $\inf_{\RN}|\Psi_\eps(t)|>0$, one may define 
$(b_\eps,v_\eps)(t)$ explicitely as a function of $\Psi_\eps(t)$. In fact, to 
prove that $\Psi_\eps$ and $(b_\eps,v_\eps)$ are globally defined, 
and to establish Theorems \ref{thm : comparison 1} and \ref{thm : comparison 2} it suffices to show 
that $\|(b_\eps,v_\eps)\|_{H^{s+1}\times H^s}$ remains bounded. Moreover, to obtain
 the bound $\||\Psi_\eps(t)|^2-1\|_{\infty}<1/2$, it suffices to show 
that \eqref{est:binf} holds as long as $b_\eps$ is defined. 

Due to the presence of higher order derivatives in the right-hand sides in \eqref{eq : b-v}, controlling $\|(b_\eps,v_\eps)\|_{H^{s+1}\times H^s}$ 
is however a difficult issue. As in \cite{BDS}, this control will be carried out 
by incorporating  the equation satisfied  by 
  $\nabla \ln(\rho_\eps^2)$. More precisely, we focus on the new 
variable $(b_\eps,z_\eps)$, where 
\begin{equation*}
 z_\eps= v_\eps-i\nabla
\ln(\rho_\eps^2) =\nabla \big( 2\varphi_\eps-i\ln(\rho_\eps^2)\big)\in \C^N.
\end{equation*}
We remark that $(b_\eps,z_\eps)$ is well-suited to our analysis since
\begin{equation*}
 E_\eps(\Psi_\eps)= \frac{1}{8}\left(\|b_\eps\|_{L^2}^2+\|z_\eps\|_{L^2((1+\eps b/\sqrt{2})dx)}^2\right).
\end{equation*}
Moreover, there exists a constant $C=C(s,N)$ such that\footnote{See \eqref{sim : DI} below.}
\begin{equation*} 
 C^{-1}\|(b_\eps,z_\eps)\|_{H^s}\leq \|(b_\eps,v_\eps)\|_{H^s}+\eps \|b_\eps\|_{H^{s+1}}\leq C\|(b_\eps,z_\eps)\|_{H^s}.
\end{equation*}

\medskip

From now on we will sometimes omit the subscript $\eps$ for more clarity in the notations.

\medskip

The equations for $(b,z)$ are given in the following

\begin{proposition}
 \label{prop : equation b-z}
Let $s\geq 2$, $T_0>0$ and $\Psi$ be a solution to \eqref{eq : cgle} on $[0,T_0]$ satisfying
 $$\inf_{(t,x)\in [0,T_0]\times \RN} |\Psi(t,x)|\geq m>0$$ and such that $(b,v)\in C^1([0,T_0],H^{s+1}\times H\sp s)$.
Then\footnote{Here $\dsp \langle z,z\rangle=\sum_{i=1}^N z_i^2$, where $z=(z_1,\ldots,z_N)\in \C^N$.}
\begin{equation*}
\begin{cases} 
\begin{split} \dt b+\frac{\sqrt{2}}{\eps}\mathrm{div} \mathrm{Re} z=\ke \Big(
-(\frac{\sqrt{2}}{\eps}+b)&\mathrm{div}(\mathrm{Im}
z)-\frac{1}{2}(\frac{\sqrt{2}}{\eps}+b)\mathrm{Re}\langle
z,z\rangle\\
&-\frac{\sqrt{2}}{\eps}(\frac{\sqrt{2}}{\eps}+b)b\Big)
-\mathrm{div} (b\mathrm{Re}z)\end{split}\\ \\
\dsp \dt z+\frac{\sqrt{2}}{\eps} \nabla b=(\ke + i)\D z+\frac{-1+\ke i}{2}\nabla \langle z,z\rangle
+\ke \frac{\sqrt{2}}{\eps}i \nabla b.
\end{cases}
\end{equation*}
\end{proposition}
Dealing with $(b,z)$ instead of $(b,v)$ presents many advantages when 
computing energy estimates. Indeed, in contrast with System \eqref{eq : b-v} for $(b,v)$, the equations 
for $(b,z)$ involve only non linear first-order quadratic terms and a linear second-order operator $(\ke+i)\D z$. This is due to the identity
\begin{equation*}
\label{eq : gain one derivative}
\frac{\eps}{\sqrt{2}}\nabla b=-\vor \text{Im}z,
\end{equation*}
 which enables to save one derivative.  

For the Gross-Pitaevskii equation \eqref{eq:gross-pita}, the energy estimates 
performed in \cite{BDS} for $(b,z)$ involve a family of semi-norms with a suitable weight 
\begin{equation*}
 \Gamma^k(b,z):=\int_{\RN} |D^kb|^2+\int_{\RN}\vor |D^k z|^2,\quad k=0,\ldots,s.
\end{equation*}
In particular, we have the remarkable identity
$$\Gamma^0(b,z)=8E_\eps(\Psi),$$
which in fact was the principal motivation to add the imaginary part of $z$. Moreover we remark that $\Gamma^k(b,z)$ and $\|(D^kb,D^kz)\|_{L^2}^2$ are comparable as long as $|\Psi|$ is close to one. 

For the complex Ginzburg-Landau equation \eqref{eq : cgle} we will partly rely on the estimates already stated in \cite{BDS} to establish the following
\begin{proposition}
\label{prop : estimate-energy 2}
Let $s >N/2$ and $T_0>0$. Let $\Psi$ be a solution to \eqref{eq : cgle} on $[0,T_0]$ such that
$$\||\Psi|^2-1\|_{L^\infty([0,T_0]\times \RN)}< \frac{1}{2}$$
 and such that $(b,z)\in C^1([0,T_0], H^{s+1})$. There exists a constant $K=K(s,N)$ depending only on $s$ and $N$ such that for $1\leq k\leq s$ and $t\in[0,T_0]$ 
\begin{equation*}
\begin{split}
\frac{d}{dt} \big(\Gamma ^{k}&(b,z)+E_\eps(\Psi)\big)+\frac{\ke}{2}
\big(\Gamma^{k+1}(b,z)+\frac{1}{\eps^2}\Gamma^{k}(b,0)\big)\\
&\leq K\big( \lae \|b\|_{\infty}+\ke \|(b,z)\|_{\infty}^2+\|(Db,Dz)\|_{\infty}\big)\big(\Gamma^{k}(b,z)+E_\eps(\Psi)\big).
\end{split}
\end{equation*}
\end{proposition}

We further assume that $s>1+N/2$. Combining Proposition \ref{prop : estimate-energy 2} and Sobolev embedding we readily find
\begin{equation*}
 \|(b,z)(t)\|_{H^s}\leq C\|(b,z)(0)\|_{H^s}+C(\eps)\int_0^t \|(b,z)(\tau)\|_{H^s}^3\,d\tau.
\end{equation*}
This provides a first control of the norm $\|(b,z)(t)\|_{H^s}$ up to times of 
order $C(\eps)^{-1}\|(b,z)(0)\|_{H^s}^{-2}$. However, 
we need to refine this control since  $C(\eps)$ diverges as  $\eps$ tends to zero. In fact,
 one may also apply Cauchy-Schwarz inequality and 
Sobolev imbedding together 
with Proposition \ref{prop : estimate-energy 2} to infer 
an estimate for  $\|\U\|_{L_t^\infty(H\sp s)}$ in terms of the norms  $\|\U\|_{L_t\sp 2(H\sp s)}$ and $\|b\|_{L_t\sp 2(L^\infty)}$. 

\begin{proposition}
\label{prop : estimate for X 2} Under the assumptions of 
Proposition \ref{prop : estimate-energy 2}, we assume moreover that $s>1+N/2$. There exists a constant $K=K(s,N)$ depending only on $s$ and $N$ such that for $[0,T_0]$
\begin{equation*}
\begin{split}
K^{-1} \|(b,z)\|_{L_t^\infty(H\sp s)}
&\leq \|(b,z)(0)\|_{H^s}\\
& +
\lae \|\U\|_{L_t\sp 2(H\sp s)} \|b\|_{L_t\sp 2(L^\infty)}+\big(\ke \|\U\|_{L_t^\infty(H\sp s)}+1\big) \|\U\|_{L_t\sp 2(H\sp s)}^2
\end{split}
\end{equation*}
and
\begin{equation*}
 \begin{split}
K^{-1}  \ke\|&(Db,Dz)\|_{L^2_t(H^s)}^2\leq \|(b,z)(0)\|_{H^s}^2\\
&+\|\U\|_{L^\infty_t(H^s)}\big(\lae \|\U\|_{L_t\sp 2(H\sp s)} \|b\|_{L_t\sp 2(L^\infty)}
+\big(\ke \|\U\|_{L_t^\infty(H\sp s)}+1\big) \|\U\|_{L_t\sp 2(H\sp s)}^2\big).
 \end{split}
\end{equation*}
\end{proposition}

In the second step of the proofs, we will exploit 
the decreasing properties of the semi-group 
operator associated to System \eqref{eq : b-v} to derive estimates for the 
norms $\|(b,z)\|_{L^2_t(H^s)}$ and $\|b\|_{L^2_t(L^\infty)}$ in terms of $\|(b,z)\|_{L\sp \infty_t(H^s)}$. 
These estimates are summarized in the following

\begin{proposition}
 \label{prop : estimate for Y} 
Under the assumptions of Proposition \ref{prop : estimate for X 2}, there exists a constant $K=K(s,N)$ depending only on $s$ and $N$ such that for $t\in [0,T_0]$
\begin{equation*}
\begin{split}
 K^{-1}\|(b,z)\|_{L_t\sp 2(H\sp s)}
&\leq \ke^{1/2}\max(1,\lae^{-1})M_0\\
+\big(1+\eps \|\U\|_{L_t^\infty(H\sp s)}\big)&\|\U\|_{L_t\sp 2(H\sp s)}
\big(\ke^{1/2} \|\U\|_{L_t\sp 2(H\sp s)}+(\eps+\lae^{-1}) \|\U\|_{L_t^\infty(H\sp s)}\big)
\end{split}
\end{equation*}
and
\begin{equation*}
\begin{split}
K^{-1} \|b\|_{L_t\sp 2(L^\infty)}
&\leq (\eps\lae^{-1})^{1/2} M_0\\
&+\big(1+\eps \|\U\|_{L_t^\infty(H\sp s)}\big)\|\U\|_{L_t\sp 2(H\sp s)}
\eps \max(1,\lae^{-1})\|\U\|_{L_t^\infty(H\sp s)},
\end{split}
\end{equation*}
where $M_0$ is defined in Theorem \ref{thm : theorem2}.
\end{proposition}
Combining Propositions \ref{prop : estimate for X 2} and \ref{prop : estimate for Y} yields an improved estimate for  $\|(b,z)\|_{L\sp \infty_t(H^s)}$ which, in turn, leads to Theorems \ref{thm : theorem2}, \ref{thm : comparison 1} and \ref{thm : comparison 2}.

\medskip

The remainder of this work is organized in the following way. In 
Section \ref{section : cauchy} we study the Cauchy problem for \eqref{eq : cgle} and 
prove local well-posedness for $(b,z)$. Propositions \ref{prop : equation b-z}, 
\ref{prop : estimate-energy 2} and \ref{prop : estimate for X 2} are established
in Section \ref{section : proof-props}. Section \ref{section : fourier} is devoted to the proof of 
Proposition \ref{prop : estimate for Y} by means of a Fourier analysis. 
We finally turn to the proof of Theorems \ref{thm : theorem2} and \ref{thm : comparison 2} in 
Section \ref{section : proof-thms}. We omit the proof of 
Theorem \ref{thm : comparison 1}, which can be obtained with some minor modifications. 
At some places, we will rely on helpful estimates that are recalled or established in the appendix.


\section{The Cauchy problem for the complex Ginzburg-Landau equation}
\label{section : cauchy}

In this section, 
we address the Cauchy problem for \eqref{eq : cgle} in a space including the fields  $\Psi=(1+a)^{1/2}\exp(i\varphi)$, where $(a,\varphi)\in H^{s+1}(\RN)^2$ and $s+1\geq N/2$. We consider the set
\begin{equation*}
\mathcal{W}=\left\{U\in L^\infty(\RN),\quad \nabla U\in H^\infty(\RN)\quad \text{and}\quad 1-|U|^2\in L\sp 2(\RN)\right\}.
\end{equation*}

Applying a standard fixed point argument (see, e.g., the proof of Theorem 1 in \cite{miot}) and using
the Sobolev embedding $H^{s+1}\subset L^\infty$ if $s+1>N/2$, 
it can be shown the following
\begin{theoreme}
\label{thm : cauchy}
Let $s+1>N/2$ and $U_0\in \mathcal{W}$. For any $\omega_0\in H^{s+1}(\RN)$ there 
exists $T^\ast=T(U_0,\omega_0)>0$ and a unique maximal solution $$\Psi\in \{U_0\}+C([0,T^\ast),H^{s+1}(\RN))$$ to Eq. \eqref{eq : cgle} such that $\Psi(0)=U_0+\omega_0$. 

The Ginzburg-Landau energy of $\Psi$ is finite and satisfies
\begin{equation*}
E_\eps(\Psi(t))\leq E_\eps(\Psi(0)),\quad \forall t\in[0,T^\ast).
\end{equation*}

Moreover, there exists a number $C$ depending only on $E_\eps(\Psi(0))$ such that
\begin{equation*}
\|\Psi(t)-\Psi(0)\|_{L\sp 2(\RN)}\leq C\exp(Ct),\quad \forall t\in[0,T^\ast).
\end{equation*}

Finally, either $T^\ast=+\infty$ or $\dsp\limsup_{t\to T^\ast}\|\nabla \Psi(t)\|_{H\sp {s}}=+\infty$.
\end{theoreme}
We recall that $\mathcal{E}$ denotes the space of finite energy fields. 
Thanks to the already mentioned inclusion (see \cite{gallo})
\begin{equation*}
\mathcal{E}\subset\mathcal{W}+H^1(\RN),
\end{equation*}
a consequence of Theorem \ref{thm : cauchy} is the
\begin{corollaire}
\label{coro : loc}
Let $s+1>N/2$. Let $(a^0,\varphi^0)\in H^{s+1}(\RN)^2$. We assume that
$$\frac{\eps}{\sqrt{2}}\|a^0\|_{\infty}<1.$$
There exists $T_0>0$ and a unique solution $(b,v)\in C([0,T_0],H^{s+1}\times H^{s})$ to System 
\eqref{eq : b-v} with initial datum $(a^0,u^0=2\nabla \varphi^0)$. 
Moreover, there exists $\varphi\in C([0,T_0],H_{\mathrm{loc}}^1)$ such that $v=2\nabla \varphi$.

\end{corollaire}

\begin{proof}Set
$$\Psi^0(x)=\big(1+\frac{\eps }{\sqrt{2}}a^0(x)\big)^{1/2}\exp(i\varphi^0(x)).$$
By assumption on $(a^0,\varphi^0)$,  $\Psi^0$ belongs to $\mathcal{E}$ and  
\begin{equation}
\label{pert-eq:thabor}
\||\Psi^0|^2-1\|_{\infty}< 1.
\end{equation}
Since $\mathcal{E}\subset\mathcal{W}+H^1(\RN)$, we have $\Psi^0\in \{U_0\}+H^1(\RN)$ for some $U_0\in \mathcal{W}$. Using the embedding $H^{s+1}(\RN)\subset L^\infty(\RN)$, we check that 
$$\|\nabla \Psi^0\|_{H\sp {s}}\leq C(1+\|(a^0,u^0)\|_{H^{s+1}\times H\sp {s}}^2).$$ This shows that 
actually $\Psi^0\in \{U_0\}+H\sp {s+1}(\RN)$. Hence, by virtue of Theorem \ref{thm : cauchy} there exists $T^\ast>0$ and a unique maximal solution $\Psi\in \{U_0\}+C([0,T^\ast),H\sp {s+1})$ to \eqref{eq : cgle} such that $\Psi(0)=\Psi^0$.

 Next, thanks to \eqref{pert-eq:thabor} and to the inclusion $H^{s+1}(\RN)\subset L^\infty(\RN)$, there exists by 
time continuity a non trivial interval $[0,T_0]\subset [0,T^\ast)$ for which 
$$\inf_{(t,x)\in [0,T_0]\times \RN}|\Psi(t,x)|\geq m>0.$$
Consequently, we may find a lifting for $\Psi$ on $[0,T_0]$ :
$$\Psi(t,x)=\big(1+\frac{\eps}{\sqrt{2}} b(t,x)\big)^{1/2}\exp(i\varphi(t,x)),\quad \text{where}\quad \varphi\in L^2_\loc.$$ Setting then $v=2\nabla \varphi$, we determine $b$ and $v$ in a unique way through the identities
\begin{equation*}
b=\frac{\sqrt{2}}{\eps}(|\Psi|^2-1)\quad \text{and}\quad v= \frac{2}{|\Psi|^2}(\Psi \times \nabla \Psi).
\end{equation*}
In view of the regularity of $\Psi$ we have $(b,v)\in C([0,T_0],H^{s+1}\times H\sp {s})$. In addition, $(b,v)$ is 
a solution to System \eqref{eq : b-v} on $[0,T_0]$, and the conclusion follows.
\end{proof}

\section{Proofs of Propositions \ref{prop : equation b-z}, \ref{prop : estimate-energy 2} and 
\ref{prop : estimate for X 2}.}

\label{section : proof-props}
\subsection{Notations.}

We use this paragraph to fix some notations. The notation $a\cdot b$ denotes the standard scalar 
product on $\RN$ or $\R^{2N}$, which we extend to complex vectors by setting
\begin{equation*}
 z\cdot \zeta=(\text{Re}z,\text{Im} z)\cdot (\text{Re}\zeta,\text{Im}\zeta)\in \R,\quad \forall z,\zeta\in \C^N.
\end{equation*}

We define the complex product of $z=(z_1,\ldots,z_N)$ and $\zeta=(\zeta_1,\ldots,\zeta_N)\in \C^N$ by
\begin{equation*}
 \langle z,\zeta \rangle=\sum_{j=1}^Nz_j \zeta_j \in \C.
\end{equation*}
Therefore when $z=a+ib\in \C^N$ and $\zeta=x+iy\in \C^N$ with $a,b,x,y\in \RN$ we have
\begin{equation*}
\begin{split}
 \langle z,\zeta\rangle&= a\cdot x-b\cdot y+i(a\cdot y+b\cdot x)\quad \text{and} \quad  z\cdot \zeta=a\cdot x+b\cdot y.
\end{split}
\end{equation*}

With the same notations as above we finally introduce
\begin{equation*}
 \nabla z=\nabla a+i\nabla b\in \C^{N\times N}
\end{equation*}
and
\begin{equation*}
 \nabla z:\nabla \zeta=\nabla a:\nabla x+\nabla b:\nabla y\in \R,
\end{equation*}
where for $A,B\in \R^{N\times N}$ we have set $A:B=\text{tr}(A^tB)$.
 
\subsection{Proof of Proposition \ref{prop : equation b-z}.}
\label{subsection:rome}
Since $\Psi=\rho \exp(i\varphi)$ is a solution to $\eqref{eq : cgle}$, we have, with $v=2\nabla \varphi$,
\begin{equation*}
\begin{cases}
\dsp \frac{\dt \rho^2}{\rho^2}=2\ke \left(\frac{\D
\rho}{\rho}-\frac{|v|^2}{4}+\frac{1-\rho^2}{\eps^2}\right)-
\frac{\diver (\rho^2 v)}{\rho^2}\\
\dsp \dt (2\varphi)=2\left(\frac{\D
\rho}{\rho}-\frac{|v|^2}{4}+\frac{1-\rho^2}{\eps^2}\right)+\ke
\frac{\diver (\rho^2 v)}{\rho^2}.
\end{cases}
\end{equation*}
Taking the gradient in both equations we obtain
\begin{equation*}
\begin{cases}
\dsp \nabla \frac{\dt \rho^2}{\rho^2}=2\ke \nabla \frac{\D
\rho}{\rho}- \ke \nabla \frac{|v|^2}{2}+2\ke \nabla\frac{1-\rho^2}{\eps^2}-
\nabla \frac{\diver (\rho^2 v)}{\rho^2}\\
\dsp \dt v=2 \nabla \frac{\D \rho}{\rho}-\nabla
\frac{|v|^2}{2}+2\nabla \frac{1-\rho^2}{\eps^2}+\ke \nabla
\frac{\diver (\rho^2 v)}{\rho^2}.
\end{cases}
\end{equation*}
Since $\dt z=\dt v-i\nabla \frac{\dt \rho^2}{\rho^2}$, we have
\begin{equation*}
\begin{split}
\dt z&
=(1-\ke i)2\nabla \frac{\D \rho}{\rho}-(1-\ke i)\nabla
\frac{|v|^2}{2}+2(1-\ke i) \nabla
\frac{1-\rho^2}{\eps^2}+(\ke+i)\nabla \frac{\diver (\rho^2 v)}{\rho^2}.
\end{split}
\end{equation*}
Next, expanding
$$\dsp \D \ln\rho=\frac{\D \rho}{\rho} -\frac{|\nabla \rho|^2}{\rho^2},
$$
we obtain
\begin{equation*}
\begin{split}
2\nabla \frac{\D \rho}{\rho} &=\nabla \D \ln\rho^2+2\nabla |\nabla
\ln\rho|^2
=-\D \text{Im} z+\frac{1}{2} \nabla |\text{Im}z|^2.
\end{split}
\end{equation*}
On the other hand, since $v$ is a gradient we have
\begin{equation*}
\nabla \frac{\diver (\rho^2v)}{\rho^2}=\nabla \diver v+\nabla
\Big(v\cdot \frac{\nabla \rho^2}{\rho^2}\Big) =\D \text{Re}z-\nabla
\Big(\text{Im}z\cdot v\Big).
\end{equation*}
Finally, using the fact that
$$
2 \nabla \frac{1-\rho^2}{\eps^2}=-\frac{\sqrt{2}}{\eps}\nabla b,
$$
we are led to the equation for $z$ 
\begin{equation*}
\dt z=(\ke + i)\D z-\frac{1-\ke i}{2}\nabla \langle z,z\rangle
-\frac{\sqrt{2}}{\eps}(1-\ke i) \nabla b.
\end{equation*}

We next turn to the equation for $b$, recalling that $\rho^2$ verifies
\begin{equation*}
\dt \rho^2=\ke \left( 2\rho \D
\rho-\rho^2\frac{|v|^2}{2}+2\frac{\rho
^2(1-\rho^2)}{\eps^2}\right)-\diver (\rho^2v).
\end{equation*}
Expanding the expression
\begin{equation*}
2\rho \D \rho=\rho^2 \D \ln \rho^2+\frac{\rho^2}{2}
|\text{Im}z|^2=-\rho^2 \diver \text{Im}z+\frac{\rho^2}{2}
|\text{Im}z|^2,
\end{equation*}
we find
\begin{equation*}
\begin{split}
\dt \rho^2=\ke \Big( -(1+\frac{\eps}{\sqrt{2}}b)\diver &\text{Im}
z-\frac{1}{2}(1+\frac{\eps}{\sqrt{2}}b)\text{Re}\langle
z,z\rangle-2\frac{(1+\frac{\eps}{\sqrt{2}}b)\frac{\eps}{\sqrt{2}}}{\eps^2}b\Big)\\&-\diver
\big((1+\frac{\eps}{\sqrt{2}}b)\text{Re}z\big),
\end{split}
\end{equation*}
as we wanted.
\finpreuve

\subsection{Proof of Proposition \ref{prop : estimate-energy 2}.}

We present now the proof of Proposition \ref{prop : estimate-energy 2}. 
In all this paragraph, $C$  stands for a number depending only on $s$ and $N$, 
which possibly changes from a line to another. 
We will make use of the identity
\begin{equation}
 \label{eq : gain-derivative}
\frac{\eps}{\sqrt{2}}\nabla b=-\vor \text{Im}z.
\end{equation}

As we want to rely on the estimates already 
performed for the Gross-Pitaevskii equation in \cite{BDS}, it is convenient to write the equations for $(b,z)$ as follows 
\begin{equation*}
 \begin{cases} \dt b=\ke f_{\text{d}}(b,z)+f_{\text{s}}(b,z)\\
\dsp \dt z=\ke g_{\text{d}}(b,z)+g_{\text{s}}(b,z) ,
\end{cases}
\end{equation*}
where we have introduced the dissipative part
\begin{equation*}
\begin{cases}
 \dsp f_{\text{d}}(b,z)=-(\frac{\sqrt{2}}{\eps}+b)\diver (\text{Im}
z)-\frac{1}{2}(\frac{\sqrt{2}}{\eps}+b)\text{Re}\langle
z,z\rangle-\frac{\sqrt{2}}{\eps}(\frac{\sqrt{2}}{\eps}+b)b,\\
\dsp g_{\text{d}}(b,z)=\D z+\frac{i}{2}\nabla \langle z,z\rangle+i \frac{\sqrt{2}}{\eps}\nabla b
\end{cases}
\end{equation*}
and the dispersive part 
\begin{equation*}
\begin{cases}
 \dsp f_{\text{s}}(b,z)=-\diver \big((\frac{\sqrt{2}}{\eps}+b)\text{Re}z\big),\\
\dsp g_{\text{s}}(b,z)=i\D z-\frac{1}{2}\nabla \langle z,z\rangle
-\frac{\sqrt{2}}{\eps}\nabla b.
\end{cases}
\end{equation*}

Let $k\in\mathbb{N}^\ast$. We compute
\begin{equation*}
\begin{split}
 \frac{d}{dt} \G^k(b,z)&=\frac{d}{dt}\int_{\RN} \vor D^k z\cdot D^k z+ D^k b\, D^k b\\
&=2 \int_{\RN} \vor D^k z \cdot D^k \dt z+ D^kb\, D^k\dt b
+\int_{\RN} \frac{\eps \dt b}{\sqrt{2}} D^k z\cdot D ^k z\\
&= I_{\text{s}}+ I_{\text{d}},
\end{split}
\end{equation*}
where
\begin{equation*}
 \begin{split}
  I_{\text{s}}=2 \int_{\RN} \vor D^k z \cdot D^k g_{\text{s}} + D^kb\, D^k f_{\text{s}}
+\int_{\RN} \frac{\eps f_{\text{s}}}{\sqrt{2}} D^k z\cdot D ^k z
 \end{split}
\end{equation*}
and
\begin{equation*}
 \begin{split}
\ke^{-1}  I_{\text{d}}=2 \int_{\RN} \vor D^k z \cdot D^k g_{\text{d}} + D^kb\, D^k f_{\text{d}}
+\int_{\RN} \frac{\eps f_{\text{d}}}{\sqrt{2}} D^k z\cdot D ^k z.
 \end{split}
\end{equation*}

To estimate the first term $I_{\text{s}}$ we invoke  Proposition 1 in \cite{BDS} : 
\begin{equation*}
| I_{\text{s}}|\leq C (1+\eps\|b\|_{\infty}) \|(Db,Dz)\|_{L^\infty}\left(\G^k(b,z)+E_\eps(\Psi_\eps)\right),
\end{equation*}
so we only need to estimate the term $I_{\text{d}}$. Inserting the expressions of $f_{\text{d}}$ and $g_{\text{d}}$ we find 
\begin{equation*}
 I_{\text{d}}=\ke (2I+2J+K),
\end{equation*}
where
\begin{equation*}
 \begin{split}
  I&=\int_{\RN} \vor \Big(D^k z \cdot D^k \D z+\frac{1}{2} D^k z\cdot i D^k\nabla \langle z,z\rangle+\frac{\sqrt{2}}{\eps}D^k z\cdot i D^k\nabla b\Big)\\
&=I_1+I_2+I_3,\\
J&=\int_{\RN} -D^k b\, D^k \Big((\frac{\sqrt{2}}{\eps}+b)\diver (\text{Im}
z)\Big)-\frac{1}{2}D^k b \, D^k \Big((\frac{\sqrt{2}}{\eps}+b)\text{Re}\langle
z,z\rangle\Big)\\
&\hspace{3em}-D^k b\, D^k \Big(\frac{\sqrt{2}}{\eps}(\frac{\sqrt{2}}{\eps}+b)b\Big)\\
&=J_1+J_2+J_3,\\
\hspace*{-2em}\text{and}&\\
 K&=-\int_{\RN} \vor \Big(\diver (\text{Im}
z)+\frac{1}{2}\text{Re}\langle
z,z\rangle+\frac{\sqrt{2}}{\eps}b\Big)D^k z\cdot D ^k z.
\end{split}
\end{equation*}

\medskip

\textbf{Step 1}: estimate for $I_1$.\\
Integrating by parts in $I_1$, then inserting \eqref{eq : gain-derivative} we find
\begin{equation*}
\begin{split}
 I_1&=-\int_{\RN} \vor \nabla D^kz:\nabla D^kz -\frac{\eps}{\sqrt{2}}\nabla b\cdot (D^kz\cdot \nabla D^kz)\\
&= -\int_{\RN} \vor |\nabla D^k z|^2+\int_{\RN} \vor \text{Im}z\cdot (D^kz\cdot \nabla D^kz)\\
&\leq  -\int_{\RN} \vor |\nabla D^k z|^2+\int_{\RN} \vor^{{1}/{2}}|\text{Im}z||D^kz| \vor^{{1}/{2}}|\nabla D^kz|.
\end{split}
\end{equation*}
Applying Young inequality to the second term in the right-hand side, we obtain
\begin{equation*}
 I_1\leq -\frac{1}{2}\int_{\RN} \vor |\nabla D^k z|^2+\frac{1}{2}\int_{\RN} \vor|\text{Im}z|^2 |D^kz|^2, 
\end{equation*}
so finally
\begin{equation*}
 I_1\leq -\frac{1}{2}\int_{\RN} \vor  |\nabla D^k z|^2
+ C(1+\eps\|b\|_{\infty})\|\text{Im} z\|_{\infty}^2 \|z\|_{H^k}^2.
\end{equation*}

\textbf{Step 2}: estimate for $I_2$.\\
Expanding $I_2$ thanks to Leibniz formula, we obtain
\begin{equation*}
\begin{split}
I_2&=\int_{\RN} \vor D^kz\cdot D^k(i\langle z,\nabla z\rangle)\\
&=\int_{\RN}\vor D^k z\cdot i\langle z,\nabla D^k z\rangle
+\sum_{j=0}^{k-1} C_k^j\int_{\RN} \vor D^k z\cdot i\langle D^{k-j}z,D^j(\nabla z)\rangle.
\end{split}
\end{equation*}
Applying then Young inequality to the first term in the right-hand side, we infer that 
\begin{equation*}
\begin{split}
I_2\leq \frac{1}{4}\int_{\RN} \vor &|\nabla D^k z|^2+C(1+\eps\|b\|_{\infty})\|z\|_{\infty}^2\|z\|_{H^k}^2\\
+C\sum_{j=0}^{k-1}\Big|\int_{\RN} &\vor D^k z\cdot i\langle D^{k-j}z,D^j(\nabla z)\rangle\Big|.
\end{split}
\end{equation*}
For each $0\leq j\leq k-1$, we apply first Cauchy-Schwarz, then Gagliardo-Nirenberg (see Lemma \ref{lemma : gagliardo} 
in the appendix) inequalities. This yields 
\begin{equation*}
 \begin{split}
\Big|\int_{\RN} \vor D^k z\cdot i\langle D^{k-j}z,D^j(\nabla z)\rangle\Big|&\leq C(1+\eps\|b\|_{\infty})\|D^k z\|_{L^2}\| |D^{k-j}z||D^j(\nabla z)|\|_{L^2}\\
&\leq C(1+\eps\|b\|_{\infty})\|D^k z\|_{L^2}\|D z\|_{\infty}\|z\|_{H^k},
\end{split}
\end{equation*}
and we are led to
\begin{equation*}
 I_2\leq \frac{1}{4}\int_{\RN} \vor |\nabla D^k z|^2+C(1+\eps\|b\|_{\infty})(\|z\|_{\infty}^2+\|Dz\|_{\infty})\|z\|_{H^k}^2.
\end{equation*}

\textbf{Step 3}: estimate for $I_3$.\\
Since $D^k\nabla b \in \RN$ we have by definition of the complex product
\begin{equation*}
\begin{split}
 I_3&=\int_{\RN} \vor \frac{\sqrt{2}}{\eps}D^k z\cdot i D^k\nabla b
=\int_{\RN} \vor\frac{\sqrt{2}}{\eps}D^k \text{Im}z \cdot D^k\nabla b.
\end{split}
\end{equation*}
Inserting first \eqref{eq : gain-derivative} and using then Leibniz formula we get
\begin{equation*}
 \begin{split}
  I_3&=-\frac{2}{\eps^2} \int_{\RN} \vor D^k \text{Im}z\cdot D^k\big(\vor \text{Im}z\big)\\
&=-\frac{2}{\eps^2} \int_{\RN} \vor^2 |D^k\text{Im}z|^2-\frac{2}{\eps^2}\sum_{j=1}^{k}C_k^j\int_{\RN} \vor D^k \text{Im}z\cdot \big(D^j\vor D^{k-j}\text{Im}z\big).
\end{split}
\end{equation*}
Now, we observe that for each $j\geq 1$, we have 
 $$D^j\vor =\frac{\eps}{\sqrt{2}} D^jb.$$ 
Consequently, applying Young inequality to each term of the sum we find
\begin{equation*}
 \begin{split}
  I_3&\leq -\frac{1}{\eps^2} \int_{\RN}  \vor^2 |D^k\text{Im}z|^2
+C\sum_{j=1}^k \int_{\RN} |D^jb\, D^{k-j}\text{Im}z|^2,
 \end{split}
\end{equation*}
and we finally infer from Gagliardo-Nirenberg inequality that
\begin{equation*}
 I_3\leq C\left(\|b\|_{\infty}^2+\|\text{Im}z\|_{\infty}^2\right)\|(b,z)\|_{H^k}^2.
\end{equation*}
\medskip

\textbf{Step 4}: estimate for $J_1$.\\
A short calculation using \eqref{eq : gain-derivative} yields
\begin{equation*}
 \begin{split}
J_1&=-\int_{\RN} D^k b\, D^k\Big((\frac{\sqrt{2}}{\eps}+b)\diver(\text{Im}z)\Big)\\
&=-\int_{\RN} D^k b\, D^k\diver\Big((\frac{\sqrt{2}}{\eps}+b)\text{Im}z\Big)+\int_{\RN} D^k b \,D^k(\nabla b\cdot \text{Im }z)\\
&=\int_{\RN} D^k b\, D^k \diver(\nabla b)+\int_{\RN} D^k b\, D^k(\nabla b\cdot \text{Im }z).
\end{split}
\end{equation*}
After integrating by parts in the first term in the right-hand side and expanding the second term by means of Leibniz formula we obtain
\begin{equation*}
\begin{split}
J_1&=-\int_{\RN} |\nabla D^k b|^2+\int_{\RN} D^kb\, (D^k \nabla b)\cdot \text{Im}z+\sum_{j=1}^k C_k^j\int_{\RN} D^kb \,(D^{k-j}\nabla b)\cdot D^j\text{Im}z. 
 \end{split}
\end{equation*}
Next, combining Young, Cauchy-Schwarz and Gagliardo-Nirenberg inequalities we find
\begin{equation*}
 J_1\leq -\frac{1}{2}\int_{\RN} |\nabla D^k b|^2+C\|\text{Im}z\|_{\infty}^2\|b\|_{H^k}^2+C\|b\|_{H^k}\left(\|\nabla b\|_{\infty}+\|Dz\|_{\infty}\right)\|(b,z)\|_{H^k},
\end{equation*}
so that
\begin{equation*}
 J_1\leq -\frac{1}{2}\int_{\RN} |\nabla D^k b|^2+C\left(\|\text{Im}z\|_{\infty}^2+\|(\nabla b,Dz)\|_{\infty}\right)\|(b,z)\|_{H^k}^2.
\end{equation*}
\medskip

\textbf{Step 5}: estimate for $J_2$.\\
Similarly, we compute thanks to Leibniz formula 
\begin{equation*}
 \begin{split}
  J_2&=-\frac{1}{2} \int_{\RN} D^kb\,D^k\Big((\frac{\sqrt{2}}{\eps}+b)\text{Re}\langle z,z\rangle\Big)\\
&=-\frac{1}{2} \int_{\RN} D^kb(\frac{\sqrt{2}}{\eps}+b)D^k\left(\text{Re}\langle z,z\rangle\right)+\frac{1}{2}\sum_{j=1}^k C_k^j \int_{\RN} D^kb \,D^jb D^{k-j}\left(\text{Re}\langle z,z\rangle\right)\\
&=-\frac{1}{\eps\sqrt{2}}\int_{\RN} D^kb\,D^k\left(\text{Re}\langle z,z\rangle\right)
-\frac{1}{2}\int_{\RN} b D^k b D^k\left(\text{Re}\langle z,z\rangle\right)\\
&\hspace{2em}
+\frac{1}{2}\sum_{j=1}^k C_k^j\int_{\RN} D^kb\, D^jb D^{k-j}(\text{Re}\langle z,z\rangle).
 \end{split}
\end{equation*}
Invoking Young and Cauchy-Schwarz inequalities, we obtain
\begin{equation*}
 \begin{split}
J_2\leq \frac{1}{\eps^2}&\int_{\RN}|D^k b|^2+C\|\langle z,z\rangle \|_{H^k}^2\\&
+C\big(\|b\|_{\infty}\|b\|_{H^k}\|\langle z,z\rangle\|_{H^k}+\|b\|_{H^k}\sum_{j=1}^k \|D^j b D^{k-j}\langle z,z\rangle \|_{L^2}\big),
 \end{split}
\end{equation*}
so that by virtue of Lemma \ref{lemma : gagliardo},
\begin{equation*}
 J_2\leq \frac{1}{\eps^2}\int_{\RN}|D^k b|^2+C\|(b,z)\|_{\infty}^2\|(b,z)\|_{H^k}^2.
\end{equation*}

\medskip

\textbf{Step 6}: estimate for $J_3$.\\
We have
\begin{equation*}
 \begin{split}
  J_3&=-\frac{\sqrt{2}}{\eps}\int_{\RN} D^kb\,D^k\Big(b(\frac{\sqrt{2}}{\eps}+b)\Big)\\
&=-\frac{2}{\eps^2}\int_{\RN} |D^kb|^2-\frac{\sqrt{2}}{\eps}\int_{\RN} D^kbD^k(b^2),
 \end{split}
\end{equation*}
so, thanks to Cauchy-Schwarz inequality and Lemma \ref{lemma : gagliardo},
\begin{equation*}
 J_3\leq -\frac{2}{\eps^2}\int_{\RN} |D^kb|^2+\frac{C}{\eps}\|b\|_{\infty}\|b\|_{H^k}^2.
\end{equation*}

\medskip

\textbf{Step 7}: estimate for $K$.\\
We readily obtain
\begin{equation*}
 |K|\leq C(1+\eps\|b\|_{\infty})\Big(\frac{\|b\|_{\infty}}{\eps}+\|Dz\|_{\infty}+\|z\|_{\infty}^2\Big)\|z\|_{H^k}^2.
\end{equation*}

\medskip
Gathering the previous steps we obtain 
\begin{equation*}
 \begin{split}
  \frac{d}{dt} &\G^k(b,z)+\frac{\ke}{2}\G^{k+1}(b,z)+\frac{2\ke}{\eps^2}\G^k(b,0)\\
&\leq C(1+\eps\|b\|_{\infty})\Big(\ke\big(\|(b,z)\|_{\infty}^2+\eps^{-1}\|b\|_{\infty}\big)+\|(\nabla b,Dz)\|_{\infty}\Big)\|(b,z)\|_{H^k}^2,
 \end{split}
\end{equation*}
holding for any $1\leq k\leq s$. Following step by step the previous computations we readily check that it also holds for $k=0$. Finally, we have by assumption $$\frac{1}{2}\leq 1+\frac{\eps b}{\sqrt{2}}\leq \frac{3}{2} \quad \text{on}\quad [0,T_0]\times \RN,$$ from which we infer that $\|(b,z)\|_{H^k}^2\leq C \Gamma^k(b,z)$ for all $0\leq k\leq s$. Therefore the proof of  Proposition \ref{prop : estimate-energy 2} is complete.
\finpreuve

\subsection{Proof of Proposition \ref{prop : estimate for X 2}.}

To show the first inequality we add the inequalities obtained in 
Proposition \ref{prop : estimate-energy 2} for $k$ varying from $1$ to $s$. Since $1/2\leq 1+\eps b/\sqrt{2}\leq 3/2$, this yields
\begin{equation*}
\begin{split}
 \frac{d}{dt} \|(b,z)\|_{H^s}^2&\leq C (\lae \|b\|_{\infty}+
\ke \|(b,z)\|_{\infty}^2+\|(Db,Dz)\|_{\infty})\|(b,z)\|_{H^s}^2\\
&\leq C\big(\lae \|b\|_{\infty}+
(\ke \|(b,z)\|_{H^s}+1)\|(b,z)\|_{H^s}\big)\|(b,z)\|_{H^s}^2.
\end{split}
\end{equation*}
After integrating on $[0,T]$ and using Cauchy-Schwarz inequality this leads to
\begin{equation*}
 \begin{split}
 \|(b,z)(T)\|_{H^s}^2&\leq \|(b,z)(0)\|_{H^s}^2\\
+C\|(b,z)\|_{L_T^\infty(H^s)}&\Big(\lae \|b\|_{L_T^2(L^\infty)}\|(b,z)\|_{L_T^2(H^s)}
+(\ke \|(b,z)\|_{L_T^\infty(H^s)}+1)\|(b,z)\|_{L_T^2(H^s)}^2\Big),
 \end{split}
\end{equation*}
for all $T\in[0,T_0]$. Considering the supremum over $T\in[0,t]$ and applying Young inequality in the right-hand-side we find the result.

\medskip

Finally the second inequality in Proposition \ref{prop : estimate for X 2} is obtained by integrating on $[0,t]$ and using Sobolev and Cauchy-Schwarz inequalities.
\finpreuve  

\section{Proof of Proposition \ref{prop : estimate for Y}.}

\label{section : fourier}

In this paragraph again, $C$ refers to a constant depending only on $s$ and $N$ and possibly changing from a line to another. 

First, we formulate System \eqref{eq : b-v}-\eqref{pert-eq:tilde} with second members 
involving only $b$ and $z$. By the same computations as those in Paragraph \ref{subsection:rome} we find 
\begin{equation}
\begin{cases}
\label{pert-eq:dimanche}
 \dsp\dt b+\frac{\sqrt{2}}{\eps}\diver v+\frac{2\lae}{\eps}b-\ke\D b=f(b,z)\\
\dsp \dt v+\frac{\sqrt{2}}{\eps}\nabla b-\ke\D v-\frac{\eps}{\sqrt{2}}
\nabla \D b=g(b,z),
\end{cases}
\end{equation}
where $f=\tilde{f}$ and $g=\tilde{g}-\dfrac{\eps}{\sqrt{2}}\nabla \D b$ are defined by
\begin{equation}
\label{eq : f-g}
\begin{cases}
 \dsp f(b,z)=\lae\left(-\frac{1}{\sqrt{2}}(1+\frac{\eps}{\sqrt{2}}b)|z|^2-\sqrt{2}b^2\right)-\diver(b\text{Re}z)\\
\dsp g(b,z)=-\ke \nabla (\text{Re} z\cdot \text{Im}z )+ \frac{\eps}{\sqrt{2}}
\nabla \diver(b\text{Im}z)-\frac{1}{2} \nabla \text{Re}\langle z,z\rangle.
\end{cases}
\end{equation}

\subsection{Some notations and preliminary results.}
 As in \cite{BDS}, we symmetrize System \eqref{pert-eq:dimanche} by introducing the new functions
\begin{equation*}
 c=(1-\frac{\eps^2}{2}\D)^{1/2}b,\quad d=(-\D)^{-1/2}\diver v,
\end{equation*}
and
\begin{equation*}
F=(1-\frac{\eps^2}{2}\D)^{1/2}f, \quad G=(-\D)^{-1/2}\diver g.
\end{equation*}
We remark that, knowing $d$, one can retrieve $v$ since $v$ is a gradient. We have
\begin{equation}
 \label{eq : c-d}
\begin{cases}
 \dsp\dt c+\frac{2\lae}{\eps}c-\ke\D c+\frac{\sqrt{2}}{\eps}(-\D)^{1/2}(1-\frac{\eps^2}{2}\D)^{1/2}d=F\\
\dsp \dt d-\ke \D
d-\frac{\sqrt{2}}{\eps}(-\D)^{1/2}(1-\frac{\eps^2}{2}\D)^{1/2}c=G.
\end{cases}
\end{equation}

In the following, we denote by $\xi  \in \RN$ the Fourier variable, by $\hat{f}$ the Fourier transform of $f$ and by $\FF$ the inverse Fourier transform.

\medskip

In view of the definition of $(c,d)$, it is useful to introduce the frequency threshold $|\xi|\sim \eps^{-1}$. More precisely, let us fix some $R>0$ and let $\chi$ denote the characteristic function on $B(0,R)$. For $f\in L\sp 2(\R^N)$, we define
the low and high frequencies parts of $f$
$$f_l=\FF \big(\chi(\eps \xi)\hat{f}\big)\quad \text {and}\quad  f_h=\FF \big( (1-\chi(\eps \xi))\hat{f}\big),$$
 so that $\widehat{f_l}$ and $\widehat{f_h}$ are supported in $\{|\xi|\leq R\eps^{-1}\}$ 
and $\{|\xi|\geq R\eps^{-1}\}$ respectively.

\begin{lemme}
\label{lemme : prelim}
There exists $C=C(s,N,R)>0$ such that the following holds for all $0\leq m\leq s$ and $t\in [0,T_0]$: 
\begin{equation*}
\|g(t)\|_{H^m}\approx \|G(t)\|_{H^m},\quad \|f_l(t)\|_{H^m}\approx \|F_l(t)\|_{H^m}\quad \text{and}\quad  
\| (\eps \nabla f)_h(t)\|_{H^m}\approx \|F_h(t)\|_{H^m}. 
\end{equation*}
In addition,
\begin{equation*}
\begin{split}
\|v(t)\|_{H^m}\approx \|d(t)\|_{H^m},\quad \|b_l(t)\|_{H^m}\approx \|c_l(t)\|_{H^m}\quad 
\text{and}\quad  \| (\eps \nabla b)_h(t)\|_{H^m}\approx \|c_h(t)\|_{H^m}.
\end{split}
\end{equation*}
Finally,
\begin{equation*}
\begin{split}
\|(b,z)(t)\|_{H^m}&\approx \|(b,v)_l(t)\|_{H^m}+\|(\eps \nabla
b,v)_h(t)\|_{H^m}.
\end{split}
\end{equation*}
Here we have set for $f_1,f_2\in H^m$
\begin{equation*}
 \|f_1\|_{H^m}\approx \|f_2\|_{H^m}\: \: \text{if and only if}\quad  C^{-1}\|f_1\|_{H^m}\leq \|f_2\|_{H^m}\leq C\|f_1\|_{H^m}. 
\end{equation*}
\end{lemme}
\begin{proof}
For the first two statements 
it suffices to consider the Fourier transforms of the functions 
and to use their support properties. The last statement is already established in \cite{BDS}, Lemma 1.
\end{proof}
Lemma \ref{lemme : prelim} guarantees that for $0\leq m\leq s$,
\begin{equation}
\label{sim : DI}
 \|(b,v)(t)\|_{H^m}+\eps\|b(t)\|_{H^{m+1}} \approx \|(b,z)(t)\|_{H^m}\quad \text{and} \quad 
\|(b,z)(t)\|_{H^m}\approx \|(c,d)(t)\|_{H^m},
\end{equation}
therefore we have $\|(c,d)(0)\|_{H^s}\leq CM_0$, where $M_0$ is defined in Theorem \ref{thm : theorem2}.

On the other side, when $s-1>N/2$, Sobolev embedding yields 
\begin{equation*}
 \|b_{l}(t)\|_{\infty}\leq C\|b_{l}(t)\|_{H\sp {s-1}}\leq C\|c_{l}(t)\|_{H\sp {s-1}}
\end{equation*}
and
\begin{equation*}
\begin{split}
 \|b_h(t)\|_{\infty}\leq C \|b_h(t)\|_{H\sp {s-1}}
\leq C\|(\eps \nabla b)_{h}(t)\|_{H\sp {s-1}}\leq C\|c_h(t)\|_{H\sp {s-1}}.
\end{split}
\end{equation*}
Therefore it suffices to establish the first inequality of Proposition \ref{prop : estimate for Y} for $\|(c,d)\|_{L\sp 2_t(H\sp s)}$ and the second inequality for $\|c\|_{L\sp 2_t(H\sp {s-1})}$.

\medskip

Next, we have
\begin{equation*}
 \frac{d}{dt} \begin{pmatrix} \hat{c} \\ \hat{d}\end{pmatrix}+ M(\xi)\begin{pmatrix} \hat{c} \\ \hat{d}\end{pmatrix}=\begin{pmatrix} \hat{F} \\ \hat{G}\end{pmatrix},
\end{equation*}
where 
\begin{eqnarray*}
M(\xi)=\frac{\lae}{\eps}\begin{pmatrix}
 2+\eps^2 |\xi|^2&\dsp \frac{|\xi|}{\lae}(2+\eps^2 |\xi|^2)^{1/2}
 \\\dsp-\frac{|\xi|}{\lae}(2+\eps^2 |\xi|^2)^{1/2}&\eps^2 |\xi|^2
\end{pmatrix}.
\end{eqnarray*}
By Duhamel formula we have 
\begin{equation*}
\widehat{(c,d)}(t,\xi)=e^{-tM(\xi)}\widehat{(c,d)}(0,\xi)+\int_0^t e^{-(t-\tau)M(\xi)}\widehat{(F,G)}(\tau,\xi)\,d\tau.
\end{equation*}

Our next result, which is proved in the appendix, establishes pointwise estimates for $e^{-tM(\xi)}$. 
\begin{lemme} 
\label{lemma : estimate-matrix} There 
exist positive numbers $\kappa_0$, $r$, $c$ and $C$ such that for all 
$(a,b)\in \CC$, we have for $0<\eps\leq 1$, $\kappa<\kappa_0$ and $t\geq 0$
\begin{enumerate}
 \item  If $|\xi|\leq r\lae$ then 
\begin{equation*}
\big|e^{-tM(\xi)}(a,b)\big|\leq C\exp(- \lae \eps|\xi|^2
t)\left[\exp\left(-\frac{\lae }{\eps}t\right)(|a|+|b|)+\exp\left(-\frac{c|\xi|^2}{\lae \eps}t\right)(\lae^{-1}|\xi||a|+|b|)\right].
\end{equation*}

\item If $|\xi|\geq r \lae$ then
\begin{equation*}
\big|e^{-tM(\xi)}(a,b)\big|\leq
C\exp\left(-\frac{\lae(1+\eps^2|\xi|^2)}{2\eps}t\right)(|a|+|b|).
\end{equation*}
\end{enumerate}
Here for $A=(a,b)\in \CC$ we have set $|A|=|a|+|b|$.
\end{lemme}
Lemma \ref{lemma : estimate-matrix} reveals the new frequency threshold $|\xi|\sim \lae$. We may choose $R>r$, so that $r \lae <R \eps^{-1}$. We are therefore led to split the frequency space into three regions
$$\R^N=\mathcal{R}_1\cup \mathcal{R}_2\cup \mathcal{R}_3,$$
where

\textbullet \; $\mathcal{R}_1=\{|\xi|\leq r\lae\}$ denotes the low frequencies region, in which
the semi-group is composed of a parabolic part 
($\exp(-(\lae\eps)^{-1}|\xi|^2t)$), and a damping part ($\exp(-\lae\eps^{-1}t)$).  

\textbullet \; $\mathcal{R}_2=\{r\lae\leq |\xi|\leq R\eps^{-1}\}$ denotes the intermediate 
frequencies region, in which the damping effect $\exp(-\lae \eps^{-1} t)$ is prevalent
with respect to the parabolic contribution $\exp(-\lae \eps |\xi|^2 t)$.

\textbullet \; $\mathcal{R}_3=\{|\xi|\geq R\eps^{-1}\}$ denotes the high frequencies region, in which the parabolic contribution
 is strong and dominates the damping. 

\medskip
With respect to this decomposition we introduce the small, intermediate and high frequencies parts of $f\in L\sp 2(\RN)$  as follows
\begin{equation*}
\begin{split}
 f_{s}=\FF\big({\chi}_{|\xi|\leq r \lae}\hat{f}\big),
\quad f_m= \FF\big({\chi}_{r \lae\leq |\xi|\leq R \eps^{-1}}\hat{f}\big)\quad \text{and}\quad  
f_h=\FF\big({\chi}_{|\xi|\geq R \eps^{-1}}\hat{f}\big),
\end{split}
\end{equation*}
where $\chi_E$ denotes the characteristic function on the set $E$. Note that we have $$f=f_{s}+f_m+f_h=f_l+f_h.$$

\subsection{Proof of Proposition \ref{prop : estimate for Y}.}

\label{subsection : proof-prop}

We first introduce some notations. Let
\begin{equation*}
 L(b,z)(t)=\|(1+\eps b(t))|z(t)|^2\|_{H^s}+\|b^2(t)\|_{H^s}+\|b(t) z(t)\|_{H^s}+\|\langle z,z\rangle(t)\|_{H^s}.
\end{equation*}
Next, we sort the terms in the definitions of $f(b,z)$ and $g(b,z)$ in System \eqref{eq : f-g} as follows. We set 
$$f(b,z)=\lae f_0(b,z)+f_1(b,z)$$ and
$$g(b,z)=g_1(b,z)+\eps g_2(b,z)=\nabla h_0(b,z)+\eps \nabla h_1(b,z),$$
where the subscript $j=0,1,2$ denotes the order of the derivative, so that
\begin{equation*}
 \begin{cases}
  \dsp f_0(b,z)=-\frac{1}{\sqrt{2}}(1+\frac{\eps}{\sqrt{2}}b)|z|^2 -\sqrt{2}b^2\\ \dsp f_1(b,z)=-\diver(b\text{Re}z)
\end{cases}
\end{equation*}
and
\begin{equation*}
\begin{cases}
\dsp g_1(b,z)=-\ke \nabla (\text{Re} z\cdot \text{Im}z )-1/2 \nabla \text{Re}\langle z,z\rangle=\nabla h_0(b,z)\\
\dsp g_2(b,z)=\frac{1}{\sqrt{2}}
\nabla \diver(b\text{Im}z)=\nabla h_1(b,z).
 \end{cases}
\end{equation*}

\medskip

The proof of Proposition \ref{prop : estimate for Y} relies on several lemmas which we present now separately. 

\begin{lemme}
\label{lemma : low-freq} Under the assumptions of Proposition \ref{prop : estimate for Y} we have for $T\in [0,T_0]$ 
\begin{equation*}
C^{-1}\|(c,d)_{s}\|_{L\sp 2_T(H\sp s)}\leq \ke^{1/2}\max(1,\lae^{-1})M_0+\eps \|L(b,z)\|_{L\sp 2_T }+\ke^{1/2}\|L(b,z)\|_{L^1_T}.
\end{equation*}
\end{lemme}

\begin{proof}
By virtue of Lemma \ref{lemma : estimate-matrix} we have
\begin{equation*}
\begin{split}
|\widehat{\x}_s(t,\xi)|\leq C(I(t,\xi)+J(t,\xi)),
\end{split}
\end{equation*}
where
\begin{equation*}
\begin{split}
I(t,\xi)=e^{-\frac{\lae }{\eps}t}|\widehat{\x}_s(0,\xi)|+\int_0^t e^{-\frac{\lae}{\eps}(t-\tau)}|\widehat{(F,G)}_s(\tau,\xi)|\,d\tau
\end{split}
\end{equation*}
and
\begin{equation*}
\begin{split}
J(t,\xi)&=e^{-\frac{c|\xi|^2}{\lae \eps}t}\big|(|\xi|\lae^{-1}\widehat{c_{s}}(0),\widehat{d_{s}}(0))\big|
+\int_0^t e^{-\frac{c|\xi|^2}{\lae \eps}(t-\tau)}\big|(|\xi|\lae^{-1}\widehat{F_{s}},\widehat{G_{s}})\big|\,d\tau\\
&=J_{L}(t,\xi)+J_{NL}(t,\xi).
\end{split}
\end{equation*}
We set $\check{I}=\FF I$ and $\check{J}=\FF J$, so that $\|\x_s\|_{L^2_T(H^s)}\leq C(\|\check{I}\|_{L^2_T(H^s)}+\|\check{J}\|_{L^2_T(H^s)})$.

\medskip

\textbf{First step}: estimate for $\|\check{I}\|_{L^2_T(H^s)}$.\\
Invoking Lemma \ref{lemma : estimate-exponential-decay} we obtain
\begin{equation*}
 \begin{split}\|\check{I}\|_{L\sp 2_T(H^{s})}
&\leq C\big( (\eps \lae^{-1})^{1/2}\|\x_{s}(0)\|_{H^{s}}
+\eps\lae^{-1}\|(f,g)_{s}\|_{L\sp 2_T(H^{s})}\big).
\end{split}
\end{equation*}
Let $h\in H^s$. We observe that thanks to the support properties of $\widehat{h_s}$, we have
\begin{equation*}
 \|D^k h_{s}\|_{H\sp s}\leq C\lae^{k} \|h_{s}\|_{H\sp s},\quad k\in \mathbb{N}.
\end{equation*}
Applying this inequality to the higher order derivatives $f_1,g_1$ and $g_2$, we see that
\begin{equation*}
 \begin{split}
  \|(f,g)_{s}(t)\|_{H^{s}}&\leq C(\lae +\eps \lae^2) L(b,z)(t)\leq C\lae L(b,z)(t),
 \end{split}
\end{equation*}
and we conclude that
\begin{equation}
\label{ineq:i5}
\begin{split}
\|\check{I}\|_{L\sp 2_T(H\sp s)}\leq C\big((\eps \lae^{-1})^{1/2}M_0+\eps \|L(b,z)\|_{L\sp 2_T}\big).
 \end{split}
\end{equation}
\medskip

\textbf{Second step}: estimate for $\|\check{J}\|_{L^2_T(H^s)}$.\\
We have
\begin{equation*}
 \|\check{J}\|_{L^2_T(H^s)}\leq C(\|\check{J}_L\|_{L^2_T(H^s)}+\|\check{J}_{NL}\|_{L^2_T(H^s)}).
\end{equation*}
For the linear term we obtain
\begin{equation*}
\begin{split}
\|\check{J}_L\|_{L^2_T(H^s)}&\leq \big\|(1+|\xi|^s)e^{-\frac{c|\xi|^2}{\lae \eps}t}(|\xi|\lae^{-1}|\widehat{c_{s}}(0)|+|\widehat{d_{s}}(0)|)\big\|_{L\sp 2_T(L^2)}\\
&\leq C \big\|(1+|\xi|^s)e^{-\frac{c|\xi|^2}{\lae \eps}t}|\xi|(\lae^{-1}|\widehat{c_s}(0)|+|\xi|^{-1}|\widehat{d_s}(0)|)\big\|_{L\sp 2_T(L^2)}\\
&\leq C \max(1,\lae^{-1})\big\|(1+|\xi|^s)e^{-\frac{c|\xi|^2}{\lae \eps}t}|\xi|(|\widehat{c_s}(0)|+|\widehat{\varphi_s}(0)|)\big\|_{L\sp 2_T(L^2)},
\end{split}
\end{equation*}
because $d(0)=-2(-\Delta)^{1/2} \varphi(0)$. By virtue of Lemma \ref{lemma : maximal-regularity} in the appendix, this yields
\begin{equation*}
\begin{split}\|\check{J}_L\|_{L^2_T(H^s)}
\leq C\max(1,\lae^{-1})(\eps \lae)^{1/2} \big(\|c_{s}(0)\|_{H\sp s}+\|\varphi_{s}(0)\|_{H\sp s}\big)\leq C\max(1,\lae^{-1})\ke^{1/2}M_0.
\end{split}
\end{equation*}

On the other side, Lemma \ref{lemme : prelim} yields
\begin{equation*}
 \begin{split}
\|\check{J}_{NL}\|_{L^2_T(H^s)}&\leq
  \Big\|\int_0^t(1+|\xi|^s)e^{-\frac{c|\xi|^2}{\lae \eps}(t-\tau)}\big(|\xi|\lae^{-1}|\widehat{F_{s}}|+|\widehat{G_{s}}|\big)\,d\tau\Big\|_{L\sp 2_T(L^2)}\\
&\leq \Big\|\int_0^t(1+|\xi|)^se^{-\frac{c|\xi|^2}{\lae \eps}(t-\tau)}\big(|\xi|\lae^{-1}|\widehat{f_{s}}|+|\widehat{g_{s}}|\big)\,d\tau\Big\|_{L\sp 2_T(L^2)}.
\end{split}
\end{equation*}
Inserting the expressions $f=\lae f_0+f_1$ and $g=\nabla h_0+\eps \nabla h_1$ we obtain
 \begin{equation*}
 \begin{split}
   \Big\|\int_0^t&e^{-\frac{c|\xi|^2}{\lae \eps}(t-\tau)}(1+|\xi|^s)\big(|\xi|\lae^{-1}|\widehat{f_{s}}|+|\widehat{g_{s}}|\big)\,d\tau\Big\|_{L\sp 2_T(L^2)}\\
&\leq \Big\|\int_0^te^{-\frac{c|\xi|^2}{\lae \eps}(t-\tau)}|\xi|^2
(1+|\xi|^s)\big(\lae^{-1}|\xi|^{-1}|\widehat{f_{1}}|+\eps |\xi|^{-1}|\widehat{h_{1}}|\big)\,d\tau\Big\|_{L\sp 2_T(L^2)}\\
&+\Big\|\int_0^te^{-\frac{c|\xi|^2}{\lae \eps}(t-\tau)}|\xi|
(1+|\xi|^s)\big(|\widehat{f_{0}}|+|\widehat{h_{0}}|\big)\,d\tau\Big\|_{L\sp 2_T(L^2)}.\\
 \end{split}
\end{equation*}

First, invoking Lemma \ref{lemma : maximal-regularity}, we find
\begin{equation*}
\begin{split}
 \Big\|\int_0^te^{-\frac{c|\xi|^2}{\lae \eps}(t-\tau)}&|\xi|^2(1+|\xi|^s)\big(\lae^{-1}|\xi|^{-1}|\widehat{f_{1}}|+\eps|\xi|^{-1}|\widehat{h_{1}}|\big)\,d\tau\Big\|_{L\sp 2_T(L^2)}
\\
&\leq C \eps \lae \|(1+|\xi|^s)(\lae^{-1}|\xi|^{-1}\widehat{f_{1}},\eps|\xi|^{-1}\widehat{h_1})\|_{L\sp 2_T(L^2)}\\
&\leq C\eps \lae (\lae^{-1}+\eps)\|(1+|\xi|^s)(|\widehat{b\text{Re}z}|+|\widehat{b\text{Im}z}|)\|_{L^2_T(L^2)}\\
&\leq C \eps \lae (\lae^{-1}+\eps)\|b\cdot z\|_{L\sp 2_T(H^s)}\\
&\leq C\eps \|L(b,z)\|_{L\sp 2_T}.
\end{split}
\end{equation*}

Next, we infer from Lemma \ref{lemma : maximal-regularity-2} in the appendix that 
\begin{equation*}
 \begin{split}
\Big\|\int_0^te^{-\frac{c|\xi|^2}{\lae \eps}(t-\tau)}&|\xi|(1+|\xi|^s)\big(|\widehat{f_0}|+|\widehat{h_{0}}|\big)\,d\tau\Big\|_{L\sp 2_T(L^2)}\\&\leq C (\eps \lae)^{1/2}\| (1+|\xi|^s)(|\widehat{f_0}|+|\widehat{h_0}|)\|_{L^1_T(L^2)}\\
&\leq C\ke^{1/2}\|L(b,z)\|_{L^1_T}.
 \end{split}
\end{equation*}

Gathering the previous steps and noticing 
that $(\eps\lae^{-1})^{1/2}\leq \ke^{1/2}\max(1,\lae^{-1})$, we conclude the proof of the lemma.
\end{proof}

\begin{lemme}
 \label{lemma : high-freq}
Under the assumptions of Proposition \ref{prop : estimate for Y} we have for $T \in [0,T_0]$ 
\begin{equation*}
C^{-1}\left(\|(c,d)_m\|_{L\sp 2_T(H\sp s)}+\|(c,d)_h\|_{L\sp 2_T(H^{s})}\right)\leq
(\eps\lae^{-1})^{1/2} M_0
+(\eps +\lae^{-1})\|L(b,z)\|_{L\sp 2_T}.
\end{equation*}
\end{lemme}

\begin{proof}
We divide the proof into several steps.
\medskip

\textbf{First step}: intermediate frequencies $r\lae \leq |\xi| \leq R\eps^{-1}$.\\
Another application of Lemma \ref{lemma : estimate-matrix} yields 
\begin{equation*}
 |\widehat{\x}_m(t,\xi)|\leq Ce^{-\frac{\lae }{2\eps}t}
|\widehat{\x}_m(0,\xi)|
+C\int_0^t
e^{-\frac{\lae}{2\eps}(t-\tau)}|\widehat{(F,G)}_m(\tau,\xi)|\,d\tau,
\end{equation*}
whence, according to Lemma \ref{lemma : estimate-exponential-decay}, 
\begin{equation*}
 \|\x_m\|_{L\sp 2_T(H^{s})}\leq C (\eps\lae^{-1})^{1/2} \|\x(0)\|_{H^{s}}+C \eps \lae^{-1} \|(F,G)_m\|_{L\sp 2_T(H^{s})}.
\end{equation*}
Let us set
$$(F,G)_m=\mathcal{A}_m+\mathcal{B}_m,$$
where $\mathcal{A}_m$ and $\mathcal{B}_m\in L\sp 2_T(H^{s}\times H\sp {s})$, to be determined later on, are such that $\widehat{\mathcal{A}_m}(t,\cdot)$ and $\widehat{\mathcal{B}_m}(t,\cdot)$ are compactly supported in $\big(\mathcal{R}_1\cup \mathcal{R}_2=\{|\xi|\leq R\eps^{-1}\}\big)^2$. Owing to these support properties we find
\begin{equation*}
\begin{split}
 \|(F,G)_m\|_{L\sp 2_T(H^{s})}\leq \|\mathcal{A}_m\|_{L\sp 2_T(H\sp {s})}+ \|\mathcal{B}_m\|_{L\sp 2_T(H^{s})}
&\leq C(\eps^{-1}\|\mathcal{A}_m\|_{L\sp 2_T(H\sp {s-1})}+\eps^{-2} \|\mathcal{B}_m\|_{L\sp 2_T(H^{s-2})}),
\end{split}
\end{equation*}
so finally
\begin{equation}
\label{ineq:i1}
 C^{-1}\|\x_m\|_{L\sp 2_T(H^{s})}\leq (\eps\lae^{-1})^{1/2}M_0
+ \lae^{-1} \big(\|\mathcal{A}_m\|_{L\sp 2_T(H\sp {s-1})}+\eps^{-1}\|\mathcal{B}_m\|_{L\sp 2_T(H^{s-2})}\big).
\end{equation}

\medskip

\textbf{Second step}: high frequencies $|\xi|\geq R\eps^{-1}$.\\
For the high frequencies we neglect the contribution of the damping $e^{-\frac{\lae }{2\eps}t}$ and only take the contribution of $e^{-\lae \eps |\xi|^2t}$ into account.
Exploiting again Lemma \ref{lemma : estimate-matrix} we have
\begin{equation*}
\begin{split}
  |\widehat{\x}_h(t,\xi)|&\leq Ce^{-\lae \eps |\xi|^2t}
|\widehat{\x}_h(0,\xi)|
+C\int_0^t e^{-\lae \eps|\xi|^2(t-\tau)}|\widehat{(F,G)}_h(\tau,\xi)|\,d\tau\\
&\leq C \eps |\xi|e^{-\lae \eps |\xi|^2t}
|\widehat{\x}_h(0,\xi)|
+C\int_0^t e^{-\lae \eps|\xi|^2(t-\tau)}|\widehat{(F,G)}_h(\tau,\xi)|\,d\tau,
\end{split}
\end{equation*}
where the second inequality is due to the fact that $1\leq C\eps |\xi|$ on the support of $\widehat{\x}_h$. By virtue of Lemma \ref{lemma : maximal-regularity} we obtain
\begin{equation}
\label{ineq:ref}
  \|\x_h\|_{L\sp 2_T(H^{s})}\leq C\big((\eps\lae^{-1})^{1/2} \|\x_h(0)\|_{H^{s}}+(\lae \eps)^{-1} \|(F,G)_h\|_{L\sp 2_T(H^{s-2})}\big).
\end{equation}
As in the first step, we set
 $$(F,G)_h=\mathcal{A}_h+\mathcal{B}_h,$$ where $\mathcal{A}_h$ and $\mathcal{B}_h\in L\sp 2_T(H^{s-1}\times H^{s-1})$ will be set in such a way that $\widehat{\mathcal{A}_h}(t,\cdot)$ and $\widehat{\mathcal{B}_h}(t,\cdot)$ are supported in the region $\big(\mathcal{R}_3=\{|\xi|\geq R\eps^{-1}\}\big)^2$.
Thanks to these support properties we can save one factor $\eps$ to the detriment of one derivative :
\begin{equation*}
\begin{split}
\|{(F,G)_h}\|_{L\sp 2_T(H^{s-2})}&\leq \|\mathcal{A}_h\|_{L\sp 2_T(H^{s-2})}+\|\mathcal{B}_h\|_{L\sp 2_T(H^{s-2})}
\leq C( \eps\|\mathcal{A}_h\|_{L\sp 2_T(H\sp {s-1})}+\|\mathcal{B}_h\|_{L\sp 2_T(H^{s-2})}).
\end{split}
\end{equation*}
Therefore in view of \eqref{ineq:ref} we are led to
\begin{equation}
 \label{ineq:i2}
\begin{split}
C^{-1}\|\x_h\|_{L\sp 2_T(H^{s})}\leq  (\eps\lae^{-1})^{1/2} M_0
+\lae^{-1} \big(\|\mathcal{A}_h\|_{L\sp 2_T(H\sp {s-1})}+\eps^{-1}\|\mathcal{B}_h\|_{L\sp 2_T(H^{s-2})}\big).
\end{split}
\end{equation}

\medskip

\textbf{Third step}.\\
The last step consists in choosing suitable $\mathcal{A}$ and $\mathcal{B}$. We recall that
$$(F,G)=((1-2^{-1}\eps^2\D )^{1/2}f,(-\Delta)^{1/2}\diver g),$$
and
$$f(b,z)=\lae f_0(b,z)+f_1(b,z),\quad g(b,z)=g_1(b,z)+\eps g_2(b,z).
$$
Now, for the intermediate frequencies we define 
\begin{equation*}
\begin{cases}
\mathcal{A}_{m}=\big( (1-2^{-1}\eps^2\D)^{1/2}f_{m},(-\D)^{-1/2}\diver (g_{1})_m \big)\\
\mathcal{B}_{m}=\big(0,\eps(-\D)^{-1/2}\diver (g_{2})_m\big),
\end{cases}
\end{equation*}
and for the high frequencies
\begin{equation*}
\begin{cases}
\mathcal{A}_h=\big(\lae(1-2^{-1}\eps^2\D)^{1/2}(f_{0})_h,(-\D)^{-1/2}\diver (g_{1})_h \big) \\
\mathcal{B}_h=\big((1-2^{-1}\eps^2\D)^{1/2}(f_{1})_h,\eps(-\D)^{-1/2}\diver (g_{2})_h\big).
\end{cases}
\end{equation*}

Clearly $\mathcal{A}_{m}+\mathcal{B}_m=(F,G)_m$ and $\mathcal{A}_h+\mathcal{B}_h=(F,G)_h$. Moreover, we readily check that
\begin{equation}
\label{ek:a}
\|\mathcal{A}_m\|_{H\sp {s-1}}\approx \|(f,g_1)_m\|_{H\sp {s-1}}\quad
\text{and}\quad  \|\mathcal{A}_h\|_{H\sp {s-1}}\approx \|(\lae \eps \nabla f_0,g_1)_h\|_{H\sp {s-1}}
\end{equation}
and
\begin{equation}
\label{ek:b}
 \|\mathcal{B}_m\|_{H^{s-2}}\approx \eps\|(g_{2})_m\|_{H^{s-2}}
\quad\text{and}\quad \|\mathcal{B}_h\|_{H^{s-2}}\approx \|(\eps \nabla f_1,\eps g_2)_h\|_{H^{s-2}}.
\end{equation}

On the one hand we have  
\begin{equation}
\label{ineq:g1}
 \begin{split}
\|g_1\|_{H\sp {s-1}}+\|g_2\|_{H^{s-2}}\leq C(\|z\cdot z\|_{H^s}+\|b\text{Im}z\|_{H^s})\leq C L(b,z).
\end{split}
\end{equation}
On the other hand, the support properties of $\widehat{({f}_{0})}_m$ imply that
$$\|(f_{0})_m\|_{H\sp {s-1}}\leq C\min(1,\lae^{-1})\|(f_{0})_m\|_{H\sp {s}},$$
so that
\begin{equation*}
\begin{split}
 \|f_m\|_{H\sp {s-1}}&\leq \lae \|(f_{0})_m\|_{H\sp {s-1}}+\|(f_{1})_m\|_{H^{s-1}}
\leq C(\|(f_{0})_m\|_{H^{s}}+\|(f_{1})_m\|_{H^{s-1}}),\end{split}\end{equation*}
and finally
\begin{equation}
\label{ineq:f1}
\|f_m\|_{H\sp {s-1}}\leq CL(b,z).
\end{equation}
Arguing similarly we obtain
\begin{equation}
\label{ineq:f2}
\lae \|(\eps \nabla f_0)_h\|_{H\sp {s-1}}\leq C\lae \eps \|f_0\|_{H^{s}}\leq CL(b,z)
\end{equation}
and
\begin{equation}
\label{ineq:f3}
 \|(\eps \nabla f_1)_h\|_{H^{s-2}}\leq \eps\|f_1\|_{H^{s-1}}\leq C \eps L(b,z).
\end{equation}

We infer from \eqref{ek:a}, \eqref{ineq:g1}, \eqref{ineq:f1} and \eqref{ineq:f2} that
\begin{equation}
\label{ineq:i3}
 \|\mathcal{A}_m\|_{H\sp {s-1}}+\|\mathcal{A}_h\|_{H\sp {s-1}}\leq C L(b,z).
\end{equation}
Moreover \eqref{ek:b}, \eqref{ineq:g1} and \eqref{ineq:f3} yield
\begin{equation}
\label{ineq:i4}
 \|\mathcal{B}_m\|_{H\sp {s-2}}+\|\mathcal{B}_h\|_{H\sp {s-2}}\leq C \eps L(b,z),
\end{equation}
so that the conclusion of Lemma \ref{lemma : high-freq} finally follows from \eqref{ineq:i1}, \eqref{ineq:i2}, \eqref{ineq:i3} and \eqref{ineq:i4}.
\end{proof}

Next, in order to establish the second part of Proposition \ref{prop : estimate for Y} involving the norm $\|b\|_{L\sp 2(L^\infty)}$, we show the following analogs of Lemmas \ref{lemma : low-freq} and \ref{lemma : high-freq} 
involving $\|c\|_{L\sp 2(H^{s-1})}$.

\begin{lemme}
 \label{lemma : low+high freq}
Under the assumptions of Proposition \ref{prop : estimate for Y} we have for $T\in [0,T_0]$
\begin{equation*}
C^{-1} \|c\|_{L\sp 2_T(H^{s-1})}\leq (\eps\lae^{-1})^{1/2} M_0+ \eps \max(1,\lae^{-1}) \|L(b,z)\|_{L\sp 2_T}.
\end{equation*}
\end{lemme}

\begin{proof}
We closely follow the proofs of Lemmas \ref{lemma : low-freq} and \ref{lemma : high-freq}, handling again the regions $\mathcal{R}_1$, $\mathcal{R}_2$ and $\mathcal{R}_3$ separately.

\medskip

\textbf{First step}: low frequencies $|\xi|\leq r\lae$.\\
For low frequencies one may even improve the estimates given by Lemma \ref{lemma : estimate-matrix} for the semi-group acting on $c$. Indeed, according to identity \eqref{eq : decomposition} stated in the proof of Lemma \ref{lemma : estimate-matrix}, we get the bound
\begin{equation*}
 |\widehat{c_{s}}(t,\xi)|\leq C(I(t,\xi)+J(t,\xi)),
\end{equation*}
where
\begin{equation*}
\begin{split}
I(t,\xi)=e^{-\frac{\lae }{2\eps}t}\big|\widehat{\x}_s(0,\xi)\big|+\int_0^t e^{-\frac{\lae}{2\eps}(t-\tau)}|\widehat{(F,G)}(\tau,\xi)|\,d\tau
\end{split}
\end{equation*}
and
\begin{equation*}
\begin{split}
J(t,\xi)&=e^{-\frac{c|\xi|^2}{\lae \eps}t}\big|(|\xi|^2\lae^{-2}\widehat{c_{s}},|\xi|\lae^{-1}\widehat{d_{s}})(0)\big|
+\int_0^te^{-\frac{c|\xi|^2}{\lae \eps}(t-\tau)}|(|\xi|^2\lae^{-2}\widehat{F_{s}},|\xi| \lae^{-1}\widehat{G_{s}})|\,d\tau\\
&=J_L(t,\xi)+J_{NL}(t,\xi).
\end{split}
\end{equation*}
Here again we set $\check{I}=\FF I$ and $\check{J}=\FF J$.
In view of the first step in the proof of Lemma \ref{lemma : low-freq} (see \eqref{ineq:i5}) we already know that \begin{equation*}
 \begin{split}
 \|\check{I}\|_{L\sp 2_T(H^{s})}
&\leq C\big((\eps\lae^{-1})^{1/2} M_0+\eps \|L(b,z)\|_{L\sp 2_T}\big).
 \end{split}
\end{equation*}

Next, since $|\xi|\lae^{-1}\leq r$ we have
\begin{equation*}
\begin{split}
\|\check{J}_L\|_{L^2_T(H^{s-1})}&\leq \big\| e^{-\frac{c|\xi|^2}{\lae \eps}t}(1+|\xi|^{s-1})\big(|\xi|^2\lae^{-2}|\widehat{c_{s}}(0)|+|\xi|\lae^{-1}|\widehat{d_{s}}(0)|\big)\big\|_{L\sp 2_T(L^2)}\\
&\leq C \lae^{-1}\big\| e^{-\frac{c|\xi|^2}{\lae \eps}t}|\xi|(1+|\xi|^{s-1})\big(|\widehat{c_{s}}(0)|+|\widehat{d_{s}}(0)|\big)\big\|_{L\sp 2_T(L^2)}\\
&\leq C \lae^{-1} (\eps \lae)^{1/2} M_0,
\end{split}
\end{equation*}
where the last inequality is a consequence of Lemma \ref{lemma : maximal-regularity}.

On the other side we have
\begin{equation*}
\begin{split}
\|\check{J}_{NL}\|_{L^2_T(H^{s-1})}&\leq
C\Big\| \int_0^t e^{-\frac{c|\xi|^2}{\lae \eps}(t-\tau)}(1+|\xi|^{s-1})\big(|\xi|^2\lae^{-2}|\widehat{F_{s}}|+|\xi| \lae^{-1}|\widehat{G_{s}}|\big)\,d\tau\Big\|_{L\sp 2_T(L^2)}\\
&\leq C\lae^{-2} \Big\| \int_0^t e^{-\frac{c|\xi|^2}{\lae \eps}(t-\tau)}|\xi|^2(1+|\xi|^{s-1})|\widehat{F_{s}}|\,d\tau\Big\|_{L\sp 2_T(L^2)}\\
&+C\lae^{-1}\Big\| \int_0^t e^{-\frac{c|\xi|^2}{\lae \eps}(t-\tau)}|\xi|^2(1+|\xi|^{s-1})|\xi|^{-1}|\widehat{G_{s}}|\,d\tau\Big\|_{L\sp 2_T(L^2)}.
\end{split}
\end{equation*}
Applying Lemma \ref{lemma : maximal-regularity} to each term we obtain
\begin{equation*}
\begin{split}
\|\check{J}_{NL}\|_{L^2_T(H^{s-1})}&\leq C(\lae \eps\lae^{-2}\|F_{s}\|_{L\sp 2_T(H\sp {s-1})}+\lae \eps\lae^{-1} \|D^{-1} G_{s}\|_{L\sp 2_T(H\sp {s-1})})\\
&\leq C(\eps\lae^{-1}\|f_s\|_{L\sp 2_T(H^{s-1})}+\eps \|D^{-1}g_s\|_{L^2_T(H^{s-1})})\\
&\leq C\eps \|L(b,z)\|_{L\sp 2_T}.
\end{split}
\end{equation*}
We have used the support properties of $f_s$ in the last inequality above.

Finally, we gather the previous inequalities to find
\begin{equation}
\label{ineq:c1}
C^{-1}\|c_{s}\|_{L_T\sp 2(H\sp {s-1})}\leq (\eps \lae^{-1})^{1/2} M_0+\eps \|L(b,z)\|_{L\sp 2_T}.
\end{equation}

\medskip

\textbf{Second step}: intermediate frequencies $r\lae \leq |\xi|\leq R\eps^{-1}$.\\
In contrast with the previous step, we may here imitate the first step of the proof of Lemma \ref{lemma : high-freq}, estimating the $H\sp {s-1}$ norm instead :
\begin{equation*}
\begin{split}
\|c_m\|_{L\sp 2_T(H\sp {s-1})}&\leq  \|\x_m\|_{L\sp 2_T(H\sp {s-1})}
\leq C \big((\eps \lae^{-1})^{1/2} \|(c,d)(0)\|_{H\sp {s-1}}+\eps \lae^{-1} \|(F,G)_m\|_{L\sp 2_T(H\sp {s-1})}\big).
\end{split}
\end{equation*}
Recalling that $(F,G)_m=\mathcal{A}_m+\mathcal{B}_m$, where $\widehat{\mathcal{A}_m}$ and $\widehat{\mathcal{B}_m}$ are compactly supported in the region $\{|\xi|\leq R\eps^{-1}\}$, we obtain
\begin{equation*}
\|(F,G)_m\|_{L\sp 2_T(H\sp {s-1})}\leq \|\mathcal{A}_m\|_{L\sp 2_T(H\sp {s-1})}+\|\mathcal{B}_m\|_{L\sp 2_T(H\sp {s-1})}\leq \|\mathcal{A}_m\|_{L\sp 2_T(H\sp {s-1})}+C\eps^{-1}\|\mathcal{B}_m\|_{L\sp 2_T(H\sp {s-2})}.
\end{equation*}
In view of the third step of the proof of Lemma \ref{lemma : high-freq} (see \eqref{ineq:i3} and \eqref{ineq:i4}) we get
\begin{equation*}
 \|(F,G)_m\|_{H\sp {s-1}}\leq CL(b,z)
\end{equation*}
and we conclude that
\begin{equation}
\label{ineq:c2}
C^{-1}\|c_m\|_{L_T\sp 2(H\sp {s-1})}\leq (\eps \lae^{-1})^{1/2}M_0+\eps \lae^{-1} \|L(b,z)\|_{L\sp 2_T}.
\end{equation}

\medskip
\textbf{Third step}: high frequencies $|\xi|\geq R\eps^{-1}$.\\
With $(F,G)_h=\mathcal{A}_h+\mathcal{B}_h$ we obtain, arguing exactly as in the second step of the proof of 
Lemma \ref{lemma : high-freq}, the analog of \eqref{ineq:ref}: 
\begin{equation*}
\begin{split}
 \|\x_h\|_{L\sp 2_T(H\sp {s-1})}
&\leq C\big((\eps \lae^{-1})^{1/2}M_0+ (\eps \lae)^{-1} \|(F,G)_h\|_{L\sp 2_T(H\sp {s-3})}\big)\\
&\leq C\big((\eps \lae^{-1})^{1/2}M_0+  \lae^{-1} \|(F,G)_h\|_{L\sp 2_T(H\sp {s-2})}\big)\\
&\leq C\big((\eps \lae^{-1})^{1/2}M_0+  \lae^{-1}\eps ( \|\mathcal{A}_h\|_{L\sp 2_T(H\sp {s-1})}+\eps^{-1}\|\mathcal{B}_h\|_{L\sp 2_T(H\sp {s-2})})\big).
\end{split}
\end{equation*}
Hence we infer from estimates \eqref{ineq:i3} and \eqref{ineq:i4} for $\mathcal{A}_h$ and $\mathcal{B}_h$ that
\begin{equation}
\label{ineq:c3}
C^{-1}\|c_h\|_{L\sp 2_T(H\sp {s-1})}\leq (\eps \lae^{-1})^{1/2}M_0+\eps \lae^{-1} \|L(b,z)\|_{L\sp 2 _T}.
\end{equation}

The conclusion finally follows from estimates \eqref{ineq:c1}, \eqref{ineq:c2} and \eqref{ineq:c3}.
\end{proof}

\medskip

Invoking the previous results we may now complete the

\textbf{Proof of  Proposition \ref{prop : estimate for Y}}. \\
First,  Cagliardo-Nirenberg inequality yields
\begin{equation*}
\||z|^2\|_{H^s}+\|b^2\|_{H^s}+\|b z\|_{H^s}+\|\langle z,z\rangle\|_{H^s}\leq C\|(b,z)\|_{\infty}\|(b,z)\|_{H^s}
\end{equation*}
and
\begin{equation*}
 \|\eps b|z|^2\|_{H^s}\leq C \eps \|(b,z)\|_{\infty}^2\|(b,z)\|_{H^s},
\end{equation*}
so that
\begin{equation*}
 L(b,z)\leq C(1+\eps \|(b,z)\|_{\infty})\|(b,z)\|_{\infty}\|(b,z)\|_{H^s}.
\end{equation*}
By Sobolev embedding and Cauchy-Schwarz inequality we obtain
\begin{equation*}
 \|L(b,z)\|_{L\sp 2_T}\leq  C\big(1+\eps \|(b,z)\|_{L\sp \infty_T(H^s)}\big) \|(b,z)\|_{L\sp \infty_T(H^s)}\|(b,z)\|_{L\sp 2_T(H^s)}
\end{equation*}
and
\begin{equation*}
\|L(b,z)\|_{L^1_T}\leq C\big(1+\eps \|(b,z)\|_{L\sp \infty_T(H^s)}\big)\|(b,z)\|_{L\sp 2_T(H^s)}^2.
\end{equation*}
Proposition \ref{prop : estimate for Y} finally follows from both estimates above together with Lemmas \ref{lemma : low-freq}, \ref{lemma : high-freq} and \ref{lemma : low+high freq}.
\finpreuve 

\medskip

We conclude this section with a result that will be needed in the course of the next section. We omit the proof, which is a straightforward adaptation of the proof of Lemma \ref{lemma : low+high freq}.
\begin{proposition}
\label{prop : estimate for Z}
Under the assumptions of Proposition \ref{prop : estimate for Y} we have for all $T\in [0,T_0]$
\begin{equation*}
C^{-1} \|c\|_{L\sp 2_T(H^{s})}\leq (\eps \lae^{-1})^{1/2} M_0+\lae^{-1} \|(b,z)\|_{L_T^\infty(H\sp s)}\|(b,z)\|_{L\sp 2_T(H^{s})}(1+\eps \|(b,z)\|_{L_T^\infty(H\sp s)}). 
\end{equation*}
\end{proposition}

\section{Proofs of Theorems \ref{thm : theorem2} and \ref{thm : comparison 2}.}

\label{section : proof-thms}

\subsection{Proof of Theorem \ref{thm : theorem2}.}

This paragraph is devoted to the proof of Theorem \ref{thm : theorem2}. Let
$\Psi^0 \in \mathcal{W}+ H^{s+1}$ such that
\begin{equation*}
\Psi^0=\rho^0\exp(i\varphi^0)=\big(1+\frac{\eps}{\sqrt{2}}a^0\big)^{1/2}\exp(i\varphi^0),
\end{equation*}
where $(a^0,\varphi^0)$ satisfies the assumptions of Theorem \ref{thm : theorem2}. Let $\Psi \in \mathcal{W}+C([0,T^\ast),H^{s+1})$ denote the corresponding solution to \eqref{eq : cgle} provided by Theorem \ref{thm : cauchy}. 

With $c(s,N)$ denoting a constant corresponding to the Sobolev embedding $H^s(\RN) \subset L^\infty(\RN)$, we first assume that the constant $K_1(s,N)$ in Theorem \ref{thm : theorem2} satisfies
\begin{equation}
 \label{assumption Grand Pic de la Meije}
K_1(s,N)>\sqrt{2}c(s,N).
\end{equation}
Hence
\begin{equation*}
 \||\Psi^0|^2-1\|_{\infty}=\frac{\eps}{\sqrt{2}}\|a^0\|_{\infty}<\frac{1}{2},
\end{equation*}
so that the assumptions of Corollary \ref{coro : loc} are satisfied. Let $(b,v)$ be the solution given by Corollary \ref{coro : loc} on $[0,T_0)$, with $T_0\leq T^\ast$ maximal.

\medskip

We introduce the following control function
\begin{equation}
\label{def:H}
\begin{cases}
\dsp  H(t)=\|\U\|_{L_t^\infty(H\sp s)}+\frac{\|\U\|_{L_t\sp 2(H\sp s)}}{\ke^{1/2}\max(1,\lae^{-1})}
+\frac{\|b\|_{L_t\sp 2(L^\infty)}}{(\eps \lae^{-1})^{1/2}},\\
\dsp H_0=H(0).
\end{cases}
\end{equation}

Note that, according to \eqref{sim : DI} we have 
$$H_0\leq C_1(s,N)M_0\quad \text{and}\quad   \|(b,v)(t)\|_{H^s}+\eps\|b(t)\|_{H^{s+1}}
\leq C_1(s,N) H(t),$$ 
where the constant $C_1(s,N)$ depends only on $s$ and $N$. We recall that $M_0$ is defined in Theorem \ref{thm : theorem2}. 
Increasing possibly the number $K_1(s,N)$ introduced in Theorem \ref{thm : theorem2}, 
we may assume that $C_1(s,N)<K_1(s,N)$.

\medskip

We define the stopping time
\begin{equation*}
 T_\eps=\sup \{t\in[0,T_0)\quad \text{such that} \quad H(t)< C_2(s,N)M_0\},
\end{equation*}
where $C_2(s,N)$ denotes a constant (to be specified later) satisfying
\begin{equation}
\label{assumption Doigt de Dieu}
C_1(s,N)<C_2(s,N)<K_1(s,N).
\end{equation} 
We remark that $T_\eps>0$ by continuity of $t\mapsto H(t)$. 

We next choose $\kappa_0(s,N)$ in such a way that
\begin{equation}
\label{assumption Meije Orientale}
 \kappa_0(s,N) C_2(s,N)<\frac{K_1(s,N)}{\sqrt{2}c(s,N)}.
\end{equation}
By assumption on $M_0$, this implies that for $\kappa\leq \kappa_0(s,N)$
$$
C_2(s,N) M_0< \frac{C_2(s,N)\lae}{K_1(s,N)}\leq \frac{1}{\sqrt{2}c(s,N) \eps}.
$$
In particular, since $\|(b,z)(t)\|_{H^s}\leq H(t)$, it follows that
\begin{equation}
\label{assumption : modulus}
 \||\Psi|^2(t)-1\|_{\infty} < \frac{1}{2},\quad \forall t\in [0,T_\eps).
\end{equation}

\medskip

Our next purpose is to show that $T_\eps=T_0=T^\ast=+\infty$. 

First, mollifying possibly $(a^0,u^0)$ we may asume that 
$(b,z)\in C^1([0,T_0),H^{s+1})$. By \eqref{assumption : modulus}, Propositions \ref{prop : estimate for X 2} and \ref{prop : estimate for Y} hold on $[0,T_\eps)$, so that
\begin{equation*}
\begin{split}
C(s,N)^{-1} H&\leq H_0+ H^2
\Big(\big(1+\ke H\big)\ke\max(1,\lae^{-1})^2+\big(1+\eps H\big)(\ke\max(1,\lae^{-1})^2+\eps+\lae^{-1})\Big)\\
 &\leq H_0+ H^2 \big(1+\max(\eps,\ke) H\big)(\ke\max(1,\lae^{-1})^2+\eps+\lae^{-1}).
\end{split}
\end{equation*}
Observing that
\begin{equation*}
\ke\max(1,\lae^{-1})^2+\eps+\lae^{-1}\leq 3(\ke+\lae^{-1}),
\end{equation*}
we find
\begin{equation*}
\begin{split}
C_3(s,N)^{-1}H\leq H_0+ \max(\ke,\lae^{-1})H^2 \big(1+\max(\eps,\ke) H\big).
\end{split}
\end{equation*}
 Here $C_3(s,N)$ is a constant depending only on $s$ and $N$, which can be assumed to be larger than $\max(C_1(s,N),1)$. On the other side, for $t\in [0,T_\eps]$ we have according to \eqref{assumption Doigt de Dieu} and by assumption on $M_0$
 \[\max(\eps ,\ke) H\leq \max(\eps,\ke)C_2(s,N)M_0\leq 1,\]
 so that
\begin{equation}
\label{ineq : H}
 H\leq 2C_3(s,N)\Big(M_0+ \max(\ke,\lae^{-1}) H^2\Big).
\end{equation}
At this stage we may choose the constants $C_2(s,N)$ and $K_1(s,N)$ as follows:
\begin{equation*}
C_2(s,N)=4C_3(s,N) \quad \text{and}\quad  K_1(s,N)>16 C_3(s,N)^2 \max(\sqrt{2}c(s,N),1),
\end{equation*} 
so that all conditions \eqref{assumption Grand Pic de la Meije}, \eqref{assumption Doigt de Dieu} and \eqref{assumption Meije Orientale} are met.

We now show that $T_\eps=T_0$: otherwise $T_\eps$ is finite. Hence, considering \eqref{ineq : H} at time $T_\eps$ 
we obtain
\begin{equation*}
 4C_3(s,N)M_0\leq 2 C_3(s,N)(M_0+ 16\max(\ke,\lae^{-1})C_3(s,N)^2 M_0^2),
\end{equation*}
whence
\begin{equation*}
 1\leq 16 C_3(s,N)^2\max(\ke,\lae^{-1}) M_0\leq \frac{16 C_3(s,N)^2 }{K_1(s,N)}.
\end{equation*}
By definition of $K_1(s,N)$, this leads to a contradiction, therefore $T_\eps=T_0$.

Now, since \eqref{assumption : modulus} holds on $[0,T_0)$, Corollary \ref{coro : loc} and a standard continuation argument imply that $T_0=T^\ast$. 
Invoking again \eqref{assumption : modulus} 
we easily show that 
\begin{equation*}
\|\nabla \Psi(t)\|_{H^{s}}\leq C\big(1+\|(b,v)(t)\|_{H^{s+1}\times H^s}^2\big),\quad \forall t\in[0,T^\ast)
\end{equation*}
for a constant $C$. In view of the previous estimates we obtain
\begin{equation*}
 \limsup_{t\to T^\ast} \|\nabla \Psi(t)\|_{H^s}\leq \limsup_{t\to T^\ast} C(1+H(t)^2)<\infty.
\end{equation*}
We finally conclude that $T^\ast=+\infty$ thanks to Theorem \ref{thm : cauchy}. 
\finpreuve

\subsection{Proof of Theorem \ref{thm : comparison 2}.}

We present here the proof of Theorem \ref{thm : comparison 2}. Here again, $C$ always stands for a 
constant depending only on $s$ and $N$. We define $(b_\ell,v_\ell)(t,x)=(a_\ell,u_\ell)(\eps^{-1}t,x)$, where $(a_\ell,u_\ell)$ is the solution to the linear equation \eqref{eq : a-u homogeneous} with 
initial datum $(b^0,v^0)=(a^0,u^0)$. Introducing
$(\bb,\vv)=(b-b_\ell,v-v_\ell)$, we have
\begin{equation*}
\begin{cases}
 \dsp\dt \bb+\frac{\sqrt{2}}{\eps}\diver \vv+\frac{2\lae}{\eps}\bb-\ke\D \bb=f(b,z)\\
\dsp \dt \vv+\frac{\sqrt{2}}{\eps}\nabla \bb-\ke\D \vv=g(b,z)+\frac{\eps}{\sqrt{2}}\nabla \D b.
\end{cases}
\end{equation*}
The proof of Theorem \ref{thm : comparison 2} 
relies on energy estimates, since the method used in Section \ref{section : fourier} is not 
convenient to establish uniform in time estimates.
 For $0\leq k\leq s$ we compute by integration by parts
\begin{equation*}
 \begin{split}
  \frac{1}{2}\frac{d}{dt} \|(D^k \bb, D^k \vv)(t)\|_{L\sp 2}^2&=\int_{\RN} D^k \bb\, D^k\dt \bb+ D^k \vv\cdot D^k \dt \vv\\
&=-\frac{2\lae}{\eps}\int_{\RN} |D^k \bb|^2-\ke \int_{\RN} |\nabla D^k \bb|^2-\ke\int_{\RN}|\nabla D^k \vv|^2\\
+\int_{\RN} D^k& \bb\,D^k f(b,z)
+\int_{\RN} D^k \vv\cdot D^k g(b,z)
+\frac{\eps}{\sqrt{2}}\int_{\RN}D^k \vv\cdot D^k\nabla\Delta b.
 \end{split}
\end{equation*}
We recall the decompositions $f=\lae f_0+f_1$ and 
$g=g_1+\eps g_2=\nabla h_0+\eps \nabla h_1$, where the $f_i,g_i,h_i,i=0,1,2$, 
which have been defined in Paragraph \ref{subsection : proof-prop}, are $i$-order derivatives of quadratic functions in $(b,z)$. We obtain
\begin{equation*}
  \frac{1}{2}\frac{d}{dt} \|(D^k \bb, D^k \vv)(t)\|_{L\sp 2}^2\leq I+J+K,
\end{equation*}
where
\begin{equation*}
\begin{split}
 I&=-\frac{2\lae}{\eps}\int_{\RN} |D^k \bb|^2
+\lae \int_{\RN} D^k \bb\,D^k f_0(b,z)\\
 J&=\int_{\RN} D^k \bb\,D^kf_1(b,z)+\int_{\RN}  D^k \vv \cdot D^k g_1(b,z),\\
 K&=-\ke\int_{\RN}|\nabla D^k \vv|^2+\eps \int_{\RN} D^k \vv\cdot D^k g_2(b,z)+\frac{\eps}{\sqrt{2}}\int_{\RN} D^k \vv 
\cdot D^k\nabla \D b.
\end{split}
\end{equation*}

\noindent \textbf{Estimates for $I$ and $J$.}

By virtue of Lemma \ref{lemma : gagliardo} and by Sobolev embedding we find
\begin{equation*}
 \begin{split}
  I\leq -\frac{\lae}{\eps}\int_{\RN} |D^k \bb|^2+C\eps \lae \int_{\RN} |D^k f_0|^2\leq C\ke \|f_0\|_{H^k}^2\leq 
C\ke \|(b,z)\|_{H\sp {s}}^4.
 \end{split}
 \end{equation*}

Next, Cauchy-Schwarz inequality yields
\begin{equation*}
 J\leq \|(D^k \bb, D^k \vv)\|_{L\sp 2} \|(f_1,g_1)\|_{H\sp {k}}\leq C\|(D^k \bb, D^k \vv)\|_{L\sp 2} \|(b,z)\|_{H^{k+1}}^2.
\end{equation*}

\noindent \textbf{Estimate for $K$.}

We perform an integration by parts in the last two integrals and insert
the fact that $g_2=\nabla h_1$ to obtain
\begin{equation*}
 \begin{split}
  K&=-\ke\int_{\RN}|\nabla D^k \vv|^2-\eps \int_{\RN} \diver D^k \vv\, D^k h_1-
\frac{\eps}{\sqrt{2}}\int_{\RN}\diver D^k \vv   D^k\D b\\
&\leq -\frac{\ke}{4}\int_{\RN}|\nabla D^k \vv|^2+C\frac{\eps^2}{\ke}\int_{\RN}|D^k h_1|^2
+\frac{C}{\ke}\int_{\RN}|\eps \D D^k b|^2.
\end{split}
\end{equation*}

First, by virtue of Cagliardo and Sobolev inequalities we have
\begin{equation*}
  \|D^k h_1\|_{L^2}\leq C\|(b,z)\|_{\infty}\|(b,z)\|_{H^{k+1}}\leq C\|(b,z)\|_{H^s}\|(b,z)\|_{H^{k+1}}.
\end{equation*}
Therefore:

\textbullet \, If $0\leq k\leq s-2$ we find 
\begin{equation*} 
K\leq C\ke^{-1}\eps^2\big(\|\U\|_{H^{s}}^4+ \|b\|_{H^s}^2\big).
\end{equation*}

\textbullet \, If $k=s-1$ we observe that $\| \eps \Delta D^k b\|_{L^2}\leq C\|c\|_{H\sp {s}},$
where $c=(1-\eps^2\Delta/2)^{1/2}b$ is defined in the beginning of 
Section \ref{section : fourier}. So we find
\begin{equation*}
 K\leq C\ke^{-1}\big(\eps^2\|\U\|_{H^{s}}^4+ \|c\|_{H^s}^2\big).
\end{equation*}

\textbullet \, If $k=s$, similar arguments using that $\| \eps \Delta D^k b\|_{L^2}\leq C\|c\|_{H^{s+1}}\leq C\|(b,z)\|_{H^{s+1}}$
(see \eqref{sim : DI}) yield
\begin{equation*}
 K\leq C\ke^{-1}\|(b,z)\|_{H^{s+1}}^2\big(1+\eps^2\|\U\|_{H^{s}}^2\big).
\end{equation*}

\medskip

Integrating the previous estimates for $I$, $J$ and $K$ on $[0,t]$ we find:

\textbullet\; If $0\leq k\leq s-2$,
\begin{equation*}
 \begin{split}
\|(D^k \bb, D^k \vv)(t)\|_{L\sp 2}^2&
\leq C\int_0^t \|(D^k \bb, D^k \vv)\|_{L\sp 2}\|(b,z)\|_{H^{s}}^2\,d\tau\\
&+C\int_0^t\Big((\kappa+\kappa^{-1}\eps^2)\|(b,z)\|_{H^{s}}^4+\kappa^{-1}\eps^2\|\U\|_{H^s}^2\Big) \,d\tau.
\end{split}
\end{equation*}
Appyling Young inequality to the first term in the right-hand side we infer that
\begin{equation}
\label{i:tempete en juin1}
 \begin{split}
 C^{-1} \|(D^k\bb,D^k \vv)\|_{L_t^\infty(L\sp 2)}^2 &\leq  \|(b,z)\|_{L_t^2(H\sp {s})}^4\\
+(\ke+&\ke^{-1}\eps^2)\|(b,z)\|_{L_t^\infty(H^{s})}^2 \|(b,z)\|_{L_t^2(H^{s})}^2
+\ke^{-1}\eps^2\|(b,z)\|_{L_t^2(H^{s})}^2.
 \end{split}
\end{equation}

\textbullet \, Similarly, if $k=s-1$ we have 
\begin{equation}
\label{i:tempete en juin2} 
\begin{split}
 C^{-1} \|(D^k\bb,D^k \vv)\|_{L_t^\infty(L\sp 2)}^2 &\leq  \|(b,z)\|_{L_t^2(H\sp {s})}^4\\
+(\ke+&\ke^{-1}\eps^2) \|(b,z)\|_{L_t^\infty(H\sp {s})}^2\|(b,z)\|_{L_t^2(H\sp {s})}^2
+\ke^{-1} \|c\|_{L_t^2(H\sp s)}^2.
\end{split}
\end{equation}

\textbullet \, If $k=s$ then 
\begin{equation}
\label{i:tempete en juin3}
 \begin{split}
 C^{-1} \|(D^k\bb,D^k \vv)\|_{L_t^\infty(L\sp 2)}^2 &\leq  \|(b,z)\|_{L_t^2(H\sp {s+1})}^4\\
+\ke \|(b,z)\|_{L_t^\infty(H\sp {s})}^2&\|(b,z)\|_{L_t^2(H\sp {s})}^2
+\ke^{-1}\big(1+\eps^2 \|(b,z)\|_{L_t^\infty(H\sp {s})}^2 \big)\|\U\|_{L_t^2(H^{s+1})}^2.
 \end{split}
\end{equation}

\medskip

\textbf{Proof of the uniform in time comparison estimates in Theorem \ref{thm : comparison 2}.}

\noindent We control each term in the right-hand sides in \eqref{i:tempete en juin1}, \eqref{i:tempete en juin2} and
\eqref{i:tempete en juin3} by means of the various estimates established in the previous sections. We recall that 
 the control function $H(t)$, which is defined in \eqref{def:H}, satisfies $H(t)\leq CM_0$. This controls the quantities 
$\|(b,z)\|_{L_t^2(H\sp {s})}$ and $\|(b,z)\|_{L_t^\infty(H\sp {s})}$ in terms of $M_0$. 
We use Proposition \ref{prop : estimate for Z} to estimate $\|c\|_{L^2_t(H^s)}$.
 Finally, to control $\|(b,z)\|_{L_t^2(H\sp {s+1})}$ we 
rely on the second inequality in Proposition \ref{prop : estimate for X 2}. Straightforward computations then lead to
the uniform comparison estimates in Theorem \ref{thm : comparison 2}.

\medskip

 \textbf{Proof of the time dependent comparison estimates in Theorem \ref{thm : comparison 2}.}

\noindent We go back to the previous energy estimates.

\textbullet\, If $0\leq k\leq s-2$ we apply Cauchy-Schwarz inequality in \eqref{i:tempete en juin1} to obtain
\begin{equation*}
\begin{split}
 C^{-1}\|(D^k\bb,D^k \vv)\|_{L_t^\infty(L\sp 2)}^2&\leq 
t\|(b,z)\|_{L_t\sp \infty(H\sp s)}^2\|(b,z)\|_{L_t\sp 2(H\sp s)}^2\\
&+t(\kappa+\ke^{-1}\eps^{2}) \|(b,z)\|_{L_t^\infty(H\sp s)}^4
+t\ke^{-1}\eps^2\|(b,z)\|_{L_t\sp \infty(H\sp s)}^2.
\end{split}
\end{equation*}

\textbullet\, If $k=s-1$ we similarly infer from \eqref{i:tempete en juin2} 
\begin{equation*}
\begin{split}
C^{-1} \|(D^k\bb,D^k \vv)\|_{L_t^\infty(L\sp 2)}^2&\leq 
t\|(b,z)\|_{L_t\sp \infty(H\sp s)}^2\|(b,z)\|_{L_t\sp 2(H\sp s)}^2\\
&+t(\kappa+\ke^{-1}\eps^{2}) \|(b,z)\|_{L_t^\infty(H\sp s)}^4
+t\ke^{-1}\|(b,z)\|_{L_t\sp \infty(H\sp s)}^2.
\end{split}
\end{equation*}

Using that $H(t)\leq C M_0$, the assumptions on $M_0$ as well as the 
fact that $(a_\eps,u_\eps)(t)=(b_\eps,v_\eps)(\eps t)$ we are led to the desired estimates. We omit the details. 
\finpreuve

\section{Appendix.}

In this appendix we gather some helpful results.

\subsection{Some parabolic estimates and useful tools.}

The following result is an immediate consequence of maximal regularity for the 
heat operator $e^{t\Delta}$. We refer to \cite{lady} for further details.
\begin{lemme}
 \label{lemma : maximal-regularity}
There exists $C>0$ such that for all $\lambda >0$, $a_0\in L\sp 2(\RN)$,  $a=a(s)\in L\sp 2(\R_+\times\RN)$ and $T>0$
\begin{equation*}
\|e^{\lambda t\Delta}a_0\|_{L^2_T(\dot{H}^1)}\leq \frac{C}{\sqrt{\lambda}}\|a_0\|_{L\sp 2}
\end{equation*}
and
\begin{equation*}
\left\|\Delta \int_0^t e^{\lambda(t-s)\Delta}a(s)\,ds\right\|_{L^2_T(L^2)}\leq \frac{C}{\lambda} \|a\|_{L\sp 2_T(L\sp 2)}.
\end{equation*}
\end{lemme}
We also have the following
\begin{lemme}
\label{lemma : maximal-regularity-2}
There exists $C>0$ such that for all $\lambda>0$ and  $H\in L\sp 2(\R_+\times \RN)$
\begin{equation*}
 \Big \| \int_0^t e^{\lambda(t-s)\Delta }  H(s)\,ds \Big\|_{L\sp 2_T(\dot{H}^1)}\leq \frac{C}{\sqrt{\lambda}} \int_0^T \|H(t)\|_{L\sp 2}\,dt.
\end{equation*}
 \end{lemme}
\begin{proof} We may assume that $H$ is smooth, compactly supported, and that the function
 $u(t)=\int_0^t e^{\lambda(t-s)\Delta}  H(s)\,ds$ is the smooth solution to\[\partial_t u-\lambda \Delta u=H\quad \text{and}\quad u(0)=0.\] We infer that
\begin{equation*}
 \frac{1}{2}\frac{d}{dt} \|u(t)\|_{L\sp 2}^2=\int_{\RN} u H-\lambda\int_{\RN} |\nabla u|^2,
\end{equation*}
so that
\begin{equation*}
 \lambda \|\nabla u\|_{L\sp 2_T(L\sp 2)}^2 \leq C \int_0^T \int_{\RN} |u||H|\,dt\,dx
\leq C\sup_{t\in [0,T]} \|u(t)\|_{L\sp 2}\|H\|_{L^1_T(L\sp 2)}. 
\end{equation*}
But $u(0)=0$, therefore we also have $\|u(t)\|_{L\sp 2}^2\leq C\int_0^t \int |u H|$. This yields $$\sup_{t\in [0,T]} \|u(t)\|_{L\sp 2}\leq C \|H\|_{L^1_T(L\sp 2)}$$ and the conclusion follows. 
\end{proof}

\begin{lemme}
 \label{lemma : estimate-exponential-decay}
There exists $C>0$ such that for all $\lambda>0$, $a_0\in L\sp 2$, $a\in L\sp 2(\R_+\times\RN)$ and $T>0$
\begin{equation*}
 \|e^{-\lambda t }a_0\|_{L\sp 2_T}\leq \frac{C}{\sqrt{\lambda}} \|a_0\|_{L\sp 2}
\end{equation*}
and
\begin{equation*}
 \Big\|\int_0^t e^{-\lambda (t-s)} a(s,\cdot)\,ds\Big\|_{L\sp 2_T(L\sp 2)}\leq \frac{C}{\lambda} \|a\|_{L\sp 2_T(L\sp 2)}.
\end{equation*}
\end{lemme}
\begin{proof}
We only establish the second estimate. We set $\tilde{a}(s)=a(s)$ for $s\in [0,T]$ and $\tilde{a}=0$ for $s \notin [0,T]$, so that
\begin{equation*}
 \Big\|\int_0^t e^{-\lambda(t-s)} a(s)\,ds\Big\|_{L\sp 2_T(L\sp 2)}\leq \Big\|\int_0^T e^{-\lambda(t-s)} \|\tilde{a}(s)\|_{L\sp 2(\RN)}\,ds\Big\|_{L\sp 2_T}=\|e^{-\lambda \cdot} \ast \|\tilde{a}(\cdot)\|_{L\sp 2(\RN)}\|_{L\sp 2}.
\end{equation*}
By Young inequality for the convolution, we then have
\begin{equation*}
 \Big\|\int_0^t e^{-\lambda (t-s)} a(s)\,ds\Big\|_{L\sp 2_T(L\sp 2)}\leq C\| e^{-\lambda \cdot}\|_{L^1}\|\tilde{a}\|_{L\sp 2(\R_+,L\sp 2)}.
\end{equation*}
We conclude by definition of $\tilde{a}$.
\end{proof}

We conclude this paragraph with the following result, which is a consequence of Gagliardo-Nirenberg inequality.

\begin{lemme}[see \cite{BDS}, Lemma 3]
 \label{lemma : gagliardo}
Let $k\in \mathbb{N}$ and $j\in \{0,\ldots,k\}$. There exists a constant $C(k,N)$ such that
\begin{equation*}
 \| D ^j u D^{k-j} v\|_{L^2}\leq C(k,N) \left( \|u\|_{\infty}\|D^k v\|_{L^2}+\|v\|_{\infty}\|D^k u\|_{L^2}\right)
\end{equation*}
and
\begin{equation*}
 \|u v\|_{H^k}\leq C(k,N) \left(\|u\|_{\infty}\|v\|_{H\sp k}+\|v\|_{\infty}\|u\|_{H\sp k}\right).
\end{equation*}
\end{lemme}

\subsection{Proof of Lemma \ref{lemma : estimate-matrix}.}

In all the following $C$ denotes a numerical constant. In order to simplify the notations we introduce the quantities
\begin{equation*}
\om=\eps^2 |\xi|^2\quad \text{and}\quad  \mu=\frac{1}{\lae}|\xi|\sqrt{2+\om},
\end{equation*}
and we express $M$ as follows
\begin{eqnarray*}
M=\frac{\lae}{\eps}\begin{pmatrix}
 2+\om&\mu
 \\-\mu&\om
\end{pmatrix}.
\end{eqnarray*}
First we compute the eigenvalues $\la_1$ and $\la_2$ of $M$. 
Setting
\begin{equation*}
\D=1-\mu^2,
\end{equation*}
we have
\begin{equation*}
\la_1=\frac{\lae}{\eps}(\om+1-\sqrt[\mathbb{C}]{\D})\quad \text{and}\quad 
\la_2=\frac{\lae}{\eps}(\om+1+\sqrt[\mathbb{C}]{\D}),
\end{equation*}
where $\sqrt[\mathbb{C}]{\D}$ is $\sqrt{\D}$ if $\D\geq 0$ and is $i\sqrt{-\D}$ if $\D<0$. Hence  $M=P^{-1}D P$, where $D=\text{diag}(\la_1,\la_2)$ and
\begin{equation*}
P^{-1}=\frac{1}{\mu^2-\alpha^2}\begin{pmatrix} -\mu & \alpha \\
\alpha &-\mu
\end{pmatrix},\qquad P=\begin{pmatrix} -\mu & -\alpha \\
-\alpha &-\mu
\end{pmatrix},\quad \text{with}\quad  \alpha=1+\sqrt[\C]{\D}.
\end{equation*}
Finally for all $(a,b)\in \mathbb{C}^2$ we have
\begin{equation*}
\begin{split}
e^{-tM}\begin{pmatrix} a\\ b\end{pmatrix}=P^{-1}e^{-tD}P\begin{pmatrix} a\\b
\end{pmatrix}&=\frac{1}{\mu^2-\alpha^2}\begin{pmatrix}\dsp (\mu^2a+\alpha \mu
b)e^{-\la_1t}-(\alpha^2 a+\alpha \mu b)e^{-\la_2t} \\
\dsp (\alpha \mu a+\mu^2 b)e^{-\la_2t}-(\alpha \mu a+\alpha^2 b)e^{-\la_1t}
\end{pmatrix}\\
&=\frac{e^{-\frac{\lae}{\eps}\left(1+\om\right)
t}}{\mu^2-\alpha^2}\begin{pmatrix}\dsp (\mu^2a+\alpha \mu
b)e^{t\frac{\lae}{\eps}\sqrt[\mathbb{C}]{\D}}-(\alpha^2 a+\alpha \mu b)e^{-t\frac{\lae}{\eps}\sqrt[\mathbb{C}]{\D}} \\
\dsp (\alpha \mu a+\mu^2 b)e^{-t\frac{\lae}{\eps}\sqrt[\mathbb{C}]{\D}}-(\alpha \mu
a+\alpha^2 b)e^{t\frac{\lae}{\eps}\sqrt[\mathbb{C}]{\D}}
\end{pmatrix},
\end{split}
\end{equation*}
or equivalently
\begin{equation}
\label{eq : decomposition} e^{-tM}\begin{pmatrix}
a\\b\end{pmatrix}
=e^{-\frac{\lae}{\eps}\left(1+\om\right)
t}
\left[e^{-t\frac{\lae}{\eps}\sqrt[\mathbb{C}]{\D}}\begin{pmatrix}a\\b
\end{pmatrix}+\frac{e^{t\frac{\lae}{\eps}\sqrt[\mathbb{C}]{\D}}
-e^{-t\frac{\lae}{\eps}\sqrt[\mathbb{C}]{\D}}}{\mu^2-\alpha^2}
\begin{pmatrix}\dsp \alpha \mu b+\mu^2a \\\dsp -\alpha \mu a-\alpha^2 b\end{pmatrix}\right].
\end{equation}

\medskip

 \textbf{First case} $|\xi|^2\geq 3\lae^2/8$.\\
Then $\mu^2\geq 3/4$, hence $\Delta\leq 1/4$. We need to examine the following subcases.

\medskip

\noindent \textbf{\textbullet \:$0\leq \D
\leq 1/4$.}\\ 
It follows that $\sqrt[\mathbb{C}]{\D}=\sqrt{\D}$ and $\mu^2-\alpha^2=-2(\D+\sqrt{\D})$, so that
\begin{equation*}
\begin{split}
\left|\frac{\exp(t\frac{\lae}{\eps}\sqrt{\D})-\exp(-t\frac{\lae}{\eps}\sqrt{\D})}{\mu^2-\alpha^2}\right|
&\leq \frac{\sinh\left(t
\frac{\lae}{\eps}\sqrt{\D}\right)}{\sqrt{\D}}
\leq
C \sinh\left(\frac{\lae t}{2\eps}\right),
\end{split}
\end{equation*}
where the second inequality is due to the fact that $x\mapsto\sinh(x)/x$ is an increasing function on $\R_+$. We infer that
\begin{equation}
\label{ineq : freq-1} \left|e^{-tM(\xi)}(a,b)\right| \leq
C\exp\left(-\frac{\lae }{2\eps}t\right)\exp\left(-\frac{\lae \om}{\eps}t\right)\big(|a|+|b|\big).
\end{equation}

\medskip

\noindent \textbf{\textbullet\: $-1\leq \D <0$.}\\ 
Then $\sqrt[\C]{\D}=i\sqrt{-\D}$ and
$\mu^2-\alpha^2=-2(\D+i\sqrt{-\D})$, therefore
$$|\mu^2-\alpha^2|=2\sqrt{\D^2-\D}\geq 2\sqrt{-\D}.$$ It follows that
\begin{equation*}
\begin{split}
\left|\frac{\exp(it\frac{\lae}{\eps} \sqrt{-\D})-\exp(-it\frac{\lae}{\eps} \sqrt{-\D})}{\mu^2-\alpha^2}\right|
\leq C\frac{\left|\sin\left( t\frac{\lae}{\eps} \sqrt{-\D}\right)\right|}{\sqrt{-\D}}
\leq C\frac{\lae t}{\eps} ,
\end{split}
\end{equation*}
where in the last inequality we have inserted that $|\sin x|\leq x$ for all $x\geq 0$. Since
$|\mu|\leq C$ and $|\alpha|\leq C$ this yields
\begin{equation*}
 \left|e^{-tM(\xi)}(a,b)\right| \leq C\exp\left(-\frac{\lae}{\eps}(1+\om)t\right)\left(1+\frac{\lae}{\eps}t\right)\big(|a|+|b|\big),
\end{equation*}
so finally
\begin{equation}
\label{ineq : freq-3}
\left|e^{-tM(\xi)}(a,b)\right|\leq C\exp\left(-\frac{\lae}{2\eps}(1+\om)t\right)\big(|a|+|b|\big).
\end{equation}

\medskip

\noindent \textbf{\textbullet\:$\D\leq -1$.} \\
We have \[|\mu^2-\alpha^2|=2\sqrt{\D^2-\D}\geq
2|\D|\geq C\mu^2,\] while $|\alpha|=\sqrt{1-\D}=\mu$. Hence we find
\begin{equation}
\label{ineq : freq-4} \left|e^{-tM(\xi)}(a,b)\right| \leq C\exp\left(-\frac{\lae }{\eps}(1+\om)t\right)\big(|a|+|b|\big).
\end{equation}

\medskip

\textbf{Second case} $|\xi|^2\leq 3\lae^2/8$.\\
We check that $\mu^2\leq 3(2+3\ke^2/8)/8$, therefore $1/8\leq \D\leq 1$ whenever $\ke<\kappa_0=\sqrt{8/9}.$ Moreover
 \[C^{-1}\leq |\mu^2-\alpha^2|\leq C,\quad
\alpha\leq C,\quad  \mu\leq C\quad \text{and}\quad \mu \leq C\frac{|\xi|}{\lae}.\] In addition,
\begin{equation*}
\frac{\lae}{\eps}(-1+\sqrt{\D})=-\frac{\lae}{\eps}\frac{1-\D}{1+\sqrt{\D}}=-\frac{\lae}{\eps}\frac{\mu^2}{1+\sqrt{\D}}\leq
-C\frac{\lae}{\eps}\mu^2.
\end{equation*}
Therefore in view of \eqref{eq : decomposition}
\begin{equation*}
\left|e^{-tM(\xi)}(a,b)\right| \leq
C\exp\left(-\frac{\lae}{\eps}(1+\om)t\right)\big(|a|+|b|\big)+C\exp\left(-\frac{\lae \om
}{\eps}t\right)\exp\left(-\frac{C\lae \mu^2 }{\eps}t\right)\left(\frac{|\xi|}{\lae}|a|+|b|\right).
\end{equation*}
Now, since \[C\frac{|\xi|^2}{\lae^2}\geq \mu^2=\frac{|\xi|^2}{\lae^2}(2+\om)\geq \frac{|\xi|^2}{\lae^2}\] we obtain
\begin{equation}
\label{ineq : freq-2} \left|e^{-tM(\xi)}(a,b)\right| \leq C\exp\left(-\frac{\lae \om
}{\eps}t\right)\left(\exp\left(-\frac{\lae }{\eps}t\right)+\exp\left(-\frac{C|\xi|^2}{\lae \eps}t\right)\right)
\left(\frac{|\xi|}{\lae}|a|+|b|\right).
\end{equation}

\medskip

Gathering estimates 
\eqref{ineq : freq-1} to \eqref{ineq : freq-2} and setting $r=\sqrt{3/8}$ we are led to the conclusion of the Lemma.

\finpreuve

\bigskip

\noindent \textbf{Acknowlegments.}
I warmly thank Didier Smets for his help.
This work was partly supported by the grant JC05-51279 of the Agence
Nationale de la Recherche.

\end{document}